\documentclass[12pt]{amsart}

\usepackage{amssymb}
\usepackage[all]{xy}
\usepackage[colorlinks=true, urlcolor=rltblue, citecolor=drkgreen, linkcolor=drkred] {hyperref}
\usepackage {hyperref}
\usepackage{color}
\usepackage{mathdots}
\definecolor{rltblue}{rgb}{0,0,0.4}
\definecolor{drkgreen}{rgb}{0,0.4,0}
\definecolor{drkred}{rgb}{0.5,0,0}
\usepackage{pdfsync}

\newtheorem{thm}{Theorem}[section]
\newtheorem{lemma}[thm]{Lemma}
\newtheorem{proposition}[thm]{Proposition}
\newtheorem{cor}[thm]{Corollary}
\newtheorem{theorem}[thm]{Theorem}
\newtheorem{corollary}[thm]{Corollary}

\newtheorem{fact}[thm]{Fact}

\theoremstyle{definition}
\newtheorem{definition}[thm]{Definition}

\theoremstyle{remark}
\newtheorem{observation}[thm]{Observation}
\newtheorem{obs}[thm]{Observation}
\newtheorem{remark}[thm]{Remark}
\newtheorem{example}[thm]{Example}
\newtheorem{notation}[thm]{Notation}
\newtheorem{historic}[thm]{Historic Remark}

\theoremstyle{plain}

\newcounter{contenumi}

\def\leqt{\leq_T}
\def\geqt{\geq_T}

\def\D{{\mathcal D}}

\def\upto{\mathop{\upharpoonright}}

\def\and{\mathrel{\&}}

\def\Si{\Sigma}

\newcommand\rightdate[1]{\footnotetext{  Saved: #1 \\ Compiled: \today}}
\def\A{\mathcal{A}}
\def\B{\mathcal{B}}
\def\C{\mathcal{C}}

\def\U{\mathcal{U}}
\def\F{\mathcal{F}}

\def\Q{\mathcal{Q}}
\def\om{\omega}

\def\si{\sigma}
\def\b{\beta}

\def\a{\alpha}
\def\b{\beta}

\def\A{\mathcal A}
\def\V{\mathcal V}

\def\concat{\fr}

\def\Si{\Sigma}

\def\X{{\mathcal X}}

\def\om{\omega}

\def\J{\mathcal J}

\def\implies{\Rightarrow}

\newcommand{\fr}{{}^{\smallfrown}}

\newcommand{\tree}{{\rm Tree}}
\newcommand{\forest}{{}^{\sqcup}{\tree}}
\newcommand{\mindchange}{{{}^\to}}
\newcommand{\mnskip}{{}^{{\sf p}}}
\newcommand{\treeleq}{\trianglelefteq}
\newcommand{\PSigma}{\Sigma}

\newcommand{\segment}{\subseteq}
\newcommand{\ssegment}{\subset}

\def\pass{{\sf pass}}

\begin{document}

\title[The Wadge degrees of Borel functions]{On the structure of the Wadge degrees of BQO-valued Borel functions}

\author{Takayuki Kihara}
\thanks{The first-named author was partially supported by a Grant-in-Aid for JSPS fellows.}

\address[Takayuki Kihara]{Graduate School of Informatics, Nagoya University, Japan}
\email{kihara@i.nagoya-u.ac.jp}
\urladdr{\href{http://math.mi.i.nagoya-u.ac.jp/~kihara/index.html}{math.mi.i.nagoya-u.ac.jp/$\sim$kihara}}

\author{Antonio Montalb\'an}
\thanks{The second-named author was partially supported by NSF grant DMS-0901169 and the Packard Fellowship.
}

\address[Antonio Montalb\'an]{Department of Mathematics, University of California, Berkeley, United States}
\email{antonio@math.berkeley.edu}
\urladdr{\href{http://www.math.berkeley.edu/~antonio/index.html}{www.math.berkeley.edu/$\sim$antonio}}

\rightdate{May 20, 2017 Last modified by T.}
\maketitle


%

\begin{abstract}
In this article, we give a full description of the Wadge degrees of Borel functions from $\om^\om$ to a better quasi ordering $\Q$.
More precisely, for any countable ordinal $\xi$, we show that the Wadge degrees of $\mathbf{\Delta}^0_{1+\xi}$-measurable functions $\om^\om\to\mathcal{Q}$ can be represented by countable joins of the $\xi$-th transfinite nests of $\mathcal{Q}$-labeled well-founded trees.
\end{abstract}

\tableofcontents

\section{Introduction}

In his doctorate thesis \cite{Wadge83}, Wadge proposed a notion of reducibility between sets of reals that is not only natural, but also surprisingly well behaved, as opposed to most computability theoretic reducibilities which have a rather messy structure.

\begin{definition}[Wadge \cite{Wadge83}]
Given $\A,\B\subseteq \om^\om$, we say that $\A$ is {\em Wadge reducible} to $\B$, and write $\A\leq_w\B$, if there is a continuous function $f\colon\om^\om\to\om^\om$ such that $X\in \A\iff f(X)\in \B$ for all $X\in \om^\om$.
\end{definition} 

The relation $\leq_w$ is a pre-ordering, and, as usual, it induces an equivalence $\equiv_w$ and a degree structure.
Wadge showed the Wadge degrees are semi-linearly-ordered in the sense that all anti-chains have size at most 2.
Then, Martin and Monk showed they are well-founded.
(This is all assuming $\Gamma$-determinacy when dealing with sets in a pointclass $\Gamma$.)
Furthermore, each Wadge degree is in a sense {\em natural}, and can be assigned a {\em name} using an ordinal less than $\Theta$ and a symbol from $\{\Delta,\Si,\Pi\}$ (\cite{Wesep78}; see also the Cabal volume \cite{Cabal12}), a name from which we can understand the nature of that Wadge degree.
Based on this perspective, Duparc \cite{Dup01,Dupxx} gave an explicit description of each Borel Wadge degree of a subset of $\om^\om$.

The Wadge degrees were later extended in various directions.
We can encapsulate all those extensions within the following framework:

\begin{definition}\label{def: wadge Q}
Let $(\mathcal{Q};\leq_\mathcal{Q})$ be a partial ordering.
For $\mathcal{Q}$-valued functions $\A,\B\colon\om^\om\to\mathcal{Q}$, we say that {\em $\A$ is $\Q$-Wadge reducible to $\B$} (written $\A\leq_w \B$) if there is a continuous function $\theta\colon\om^\om\to\om^\om$ such that 
\[
(\forall X\in\om^\om)\;\A(X)\leq_\mathcal{Q}\B(\theta(X)).
\]
\end{definition}

The original Wadge degrees are the case $\Q=2$ in the definition above, coding sets by their characteristic functions $\om^\om\to 2$ and viewing $2$ as the partial ordering with two incomparable elements 0 and 1.

The first extension already considered by Wadge \cite[Section 1.E]{Wadge83}, was to {\em partial} functions  $\om^\om\to \{0,1\}$, or equivalently, total functions $\om^\om\to \{\perp,0,1\}$, where $\perp$ is thought of as being below both $0$ and $1$, which are incomparable with each other.
The degree structure we obtain is also semi-well-ordered, but slightly different than the structure of the Wadge degrees. 
These degrees are connected to recent work of Day, Downey, and Westrick \cite{DDW}, as observed by Kihara \cite{Kih}.

Shortly after, Steel studied the Wadge degrees of ordinal-valued functions with domain $\om^\om$, and showed they are well-ordered (see \cite[Theorem 1]{Dup03}).
Later, together with van Engelen and Miller \cite{EMS}, they employed bqo theory to unify these results, and showed that if $\mathcal{Q}$ is better-quasi-ordered (bqo, see Definition \ref{def:bqo}), then so is the poset of the Wadge degrees of $\mathcal{Q}$-valued Borel functions.
We delay the definition of better-quasi-ordering until Definition \ref{def:bqo}, and for now let us just say that better-quasi-orderings are well-founded, have no infinite antichains, and have very good closure properties.
van Engelen, Miller and Steel's results is even more surprising than Wadge--Martin--Monk's semi-well-orderness of the 2-Wadge degrees.
Naturally defined better-quasi-orders always have a nice structure.

For a bqo $\mathcal{Q}$, the $\mathcal{Q}$-Wadge degrees are recently found to play an important role in computability theory.
The authors \cite{KMta} showed that there is a natural isomorphism between the structure of $\mathcal{Q}$-Wadge degrees and that of the ``natural'' many-one degrees of $\mathcal{Q}$-valued problems.
Hence, exploring $\mathcal{Q}$-Wadge degrees is the same thing as exploring natural $\mathcal{Q}$-many-one degrees.
The objective of this paper is to describe the structure of the $\Q$-Wadge degrees by showing that it is isomorphic to another partial ordering that is easier to visualize and understand.

In the last decade, Selivanov \cite{Sel07, Sel11} started studying the case of {\em $k$-partitions}, that is, the case when $\Q=k$, the poset with $k$ incomparable elements for finite $k$. 
Selivanov \cite{Sel07} gave a full description of the Wadge degrees of $\mathbf{\Delta}^0_2$ $k$-partitions, naming each such degree by a  $k$-labeled well-founded forest, in a way that the name describes the nature of the $k$-Wadge degree.
What he does is essentially a generalization of the Hausdroff-Kuratowski hierarchy from $k=2$ to larger $k$'s, where the structure becomes much richer.
More precisely, for a set $\mathcal{Q}$, let $\tree(\mathcal{Q})$ be the set of all $\mathcal{Q}$-labeled well-founded countable trees, and let $\forest(\mathcal{Q})$ be the set of all $\mathcal{Q}$-labeled well-founded countable forests.
Note that every such a forest $F$ can be thought of as a collection (or a disjoint union) of countably many $\mathcal{Q}$-labeled well-founded countable trees. 
Selivanov introduced a quasi-order $\treeleq$ on $\forest(k)$, where given by $S\treeleq T$ if there is a homomorphism from $S$ to $T$ which preserves inclusion of strings ($\subseteq$) and preserves labels (as defined in \ref{def: treeleq}).

\begin{theorem}[Selivanov \cite{Sel07}]
Let $k\in\om$.
The Wadge quasi-ordering on the $\mathbf{\Delta}^0_2$-measurable $k$-valued functions is isomorphic to the quasi-ordering $(\forest(k); \treeleq)$ on well-founded $k$-labeled forests.
\end{theorem}

We have recently learned that Selivanov has extended his result to the class of $\Delta^0_3$ $k$-partitions, using forests labeled with labeled trees \cite{Sel16,Sel17}.
His techniques are very different from ours.




The objective of this paper is to give a description of the Wadge degrees of Borel functions $\om^\om\to \Q$, where $\Q$ is any better-quasi-ordering (bqo), generalizing Selivanov results from $\Delta^0_3$ to all Borel functions and from finite $k$ to all bqos $\Q$.

To name the Wadge degrees of $\mathbf{\Delta}^0_{n+1}$-measurable $\Q$-valued functions, we will use trees labeled by trees labeled by trees ... labeled by $\Q$.
That is, we will define $\tree^n(\mathcal{Q})$ as $\tree(\tree(\cdots \tree(\mathcal{Q})\cdots))$ iterated $n$ times, then define $\forest^n(\Q)$ as the disjoint unions of these trees (see Section \ref{sec:lang-terms}).
We think of each forest $T\in\forest^n(\mathcal{Q})$ as a process of mind-changes which captures a natural class $\PSigma_T$ of $\mathbf{\Delta}^0_{n+1}$-measurable functions.
Based on this viewpoint, we will then define a quasi-order $\treeleq$ on $\forest^n(\Q)$ that matches Wadge reducibility on the classes of functions described by these forests.

\begin{theorem}\label{thm:main-finite-Borel-rank}
Let $\mathcal{Q}$ be a bqo.
Then, the Wadge quasi-ordering on the $\mathbf{\Delta}^0_{n+1}$-measurable $\Q$-valued functions is isomorphic to $(\forest^n(\Q),\treeleq)$.
\end{theorem}

To extend this result through the Borel hierarchy, we will introduce the $\xi$-th iterated version $\forest^\xi(\mathcal{Q})$ for each countable ordinal $\xi$, and show the following transfinite version:

\begin{theorem}\label{thm:main-infinite-Borel-rank}
Let $\mathcal{Q}$ be a bqo, and $\xi$ be a countable ordinal.
Then, the Wadge quasi-ordering on the $\mathbf{\Delta}^0_{1+\xi}$-measurable $\Q$-valued functions is isomorphic to $(\forest^\xi(\Q),\treeleq)$.
\end{theorem}

We deal with functions of finite Borel rank and prove Theorem \ref{thm:main-finite-Borel-rank} in Sections \ref{section:qo-nestedtrees}--\ref{sec:proof-1-finite}.
We will then extend those ideas to infinite Borel rank and prove Theorem \ref{thm:main-infinite-Borel-rank} in Section \ref{sec:Omega-infinite-Borel-rank}.

The main steps for the proof are as follows.
First, we need to formally define $\forest^\xi(\Q)$ and the ordering $\treeleq$.
Then, as suggested above, in Section \ref{sec:natural-pointclass}, we will assign a pointclass $\Sigma_T$ of $\Q$-valued functions to each forest $T\in\forest^\xi(\mathcal{Q})$.
For instance, $\Sigma_{\langle 0\rangle\mindchange\langle 1\rangle}$ is the class of characteristic functions of $\bf\Si^0_1$ sets, and $\Sigma_{\langle 0\rangle\mindchange\langle 1\rangle\mindchange\langle 0\rangle}$ is the class of characteristic functions of sets which are differences of two open sets.

\begin{proposition}\label{prop: measurability of Sigma T}
For every $T\in\forest^\xi(\Q)$, every function in $\Si_T$ is $\mathbf{\Delta}^0_{1+\xi}$-measurable.
\end{proposition}

These pointclasses will match the ordering $\treeleq$ on forests in the following sense:

\begin{proposition}\label{prop:order-on-trees}
For $S,T\in\forest^\xi(\Q)$, $S\treeleq T$ if and only if every $\Sigma_S$ function is Wadge reducible to some $\Sigma_T$ function.
\end{proposition}

For pointclass $\Sigma_T$, we will define a $\Sigma_T$-complete function $\Omega_T$, that is, $\Omega_T$ is in $\Sigma_T$ and any other function in $\Sigma_T$ is Wadge reducible to $\Omega_T$.

\begin{proposition}\label{prop:sublemma1}
For each $T\in\forest^\xi(\Q)$, there is a $\Sigma_T$-complete function $\Omega_T$.
\end{proposition}

We can then restate Proposition \ref{prop:order-on-trees} as $S\treeleq T\iff \Omega_S\leq_w\Omega_T$.
This gives us an embedding of $\forest^\xi(\Q)$ into the $\Delta_{1+\xi}$ $\Q$-Wadge degrees.
The last step is to show that this embedding is onto.

\begin{proposition}\label{prop:sublemma1(2)->(1)}
Every $\mathbf{\Delta}^0_{1+\xi}$-measurable function $\om^\om\to\Q$ is Wadge equivalent to a $\Sigma_T$-complete function for some $T\in\tree^\xi(\Q)$.
\end{proposition}

\section{The $\Q$-Wadge degrees}\label{sec:Q-wadge}

Let us start by describing what we knew about the structure of the $\Q$-Wadge degrees.

\subsection{The Borel hierarchy of functions}

We should be careful here, as that there are several different definitions of the Borel hierarchy (specifically, at limit ranks).
We adopt the following definition:
For $\alpha>0$, a set $\mathcal{S}\subseteq\om^\om$ is $\mathbf{\Sigma}^0_\alpha$ if $\mathcal{S}$ can be written as $\mathcal{S}=\bigcup_{n\in\om}\mathcal{S}_n$ where each $\mathcal{S}_n$ is $\mathbf{\Pi}^0_{\beta_n}$ for some $\beta_n<\alpha$.
Then, we define $\mathbf{\Pi}^0_\alpha$ and $\mathbf{\Delta}^0_\alpha$ in a usual manner.
For a countable ordinal $\xi$, and a topological space $\X$, a function $\A\colon \om^\om\to\mathcal{\X}$ is {\em $\mathbf{\Sigma}^0_\xi$-measurable} if $\A^{-1}[U]$ is $\mathbf{\Sigma}^0_\xi$ for each open set $U\subseteq\mathcal{X}$.
In particular: $\A\colon \om^\om\to\om^\om$ is $\mathbf{\Sigma}^0_\xi$-measurable if $\A^{-1}[\sigma]$ is $\mathbf{\Sigma}^0_\xi$ for each $\si\in \om^{<\om}$, where $[\si]=\{X\in\om^\om: \si\subseteq X\}$;
$\A\colon \om^\om\to\mathcal{Q}$ is $\mathbf{\Sigma}^0_\xi$-measurable if it is with respect to the discrete topology on $\mathcal{Q}$.
If $\Q$ is a discrete space, the class of total $\mathbf{\Si}^0_\xi$-measurable functions $\om^\om\to\mathcal{Q}$ is the same as that of $\mathbf{\Delta}^0_\xi$-measurable functions.
Note that the range of a Borel function from $\om^\om$ to a discrete space is countable; otherwise ZFC would prove the existence of $2^{\aleph_1}$ many pairwise different Borel subsets of $\om^\om$.
Thus, $\A:\om^\om\to\Q$ is $\mathbf{\Delta}^0_\xi$-measurable if and only if the range of $\A$ is countable, and $\A^{-1}[\{q\}]$ is $\mathbf{\Delta}^0_\xi$ for any $q\in\mathcal{Q}$.
Since we will be dealing with Borel functions $\A\colon \om^\om\to\mathcal{Q}$, they will always have countable range. 
We can thus assume from the rest of the paper that $\Q$ is actually countable, even though all the results will extend to uncountable $\Q$ for functions with countable range.


Continuous functions are exactly the $\mathbf{\Si}^0_1$-measurable functions. 
For each continuous function $G$ there is a partial computable operator $\Phi_e\colon\om^{\leq\om}\to\om^{\leq\om}$ and an oracle $C\in \om^\om$ such that $G(X)=\Phi_e(C\oplus X)$ for all $X\in \om^\om$.
Also, we will often identify a continuous function $\om^\om\to\om^\om$ with its corresponding approximation function $\om^{\leq\om}\to\om^{\leq\om}$.

For functions $\A,\B\colon\om^\om\to\om^\om$, we say that $\A$ is {\em Wadge reducible} to $\B$ if there is a continuous function $\theta$ such that $\A=\B\circ \theta$.
Note that this matches Definition \ref{def: wadge Q} if we think of $\Q$ as $\om^\om$ where every two reals are incomparable under $\leq_Q$.

%

\subsection{Wadge degrees and games}\label{sec:Wadge-game}

Wadge \cite[Theorem B8]{Wadge83} introduced a perfect-information, infinite, two-player game, known as the {\em Wadge game}, which can be used to define  Wadge reducibility.
For $\mathcal{Q}$-valued functions $\A,\B \colon \om^\om\to\mathcal{Q}$, here is the $\mathcal{Q}$-valued version $G_w(\mathcal{A},\mathcal{B})$ of the Wadge game:
At $n$-th round of the game, Player I chooses $x_n\in\om$ and II chooses $y_n\in\om\cup\{{\sf pass}\}$  (where ${\sf pass}\not\in\om$).
Eventually Players I and II produce infinite sequences $X=(x_n)_{n\in\omega}$ and $Y=(y_n)_{n\in\omega}$, respectively.
Let $Y^{{\sf p}}$ denote the result dropping all {\sf pass}es from $Y$.
We say that {\em Player II wins the game $G_w(\mathcal{A},\mathcal{B})$} if
\[\mbox{$Y^{{\sf p}}$ is an infinite sequence, and }\mathcal{A}(X)\leq_\mathcal{Q}\mathcal{B}(Y^{{\sf p}}).\]
As in Wadge \cite[Theorem B8]{Wadge83}, one can easily check that $\mathcal{A}\leq_{w}\mathcal{B}$ holds if and only if Player II wins the game $G_w(\mathcal{A},\mathcal{B})$.
We will often identify a winning strategy with a continuous function generated by it.


\subsection{Better quasi orderings}

To define bqos, we need to introduce some notation.
Let $[\om]^\om$ be the set of all strictly increasing sequences on $\om$, whose topology is inherited from $\om^\om$.
We also assume that a quasi-order $\mathcal{Q}$ is equipped with the discrete topology.
Given $X\in[\om]^\om$, let $X^-$ denote the result of dropping the first entry from $X$ (or equivalently, $X^-=X\setminus\{\min X\}$, if we  think of $X\in[\om]^\om$ as an infinite subset of $\om$).

\begin{definition}[Nash-Williams \cite{NW68}]\label{def:bqo}
A quasi-order $\mathcal{Q}$ is a {\em better-quasi-order} (abbreviated as bqo) if, for any continuous function $f\colon[\om]^\om\to\mathcal{Q}$,  there is $X\in[\om]^\om$ such that $f(X)\leq_\mathcal{Q}f(X^-)$.
\end{definition}

The formulation of the definition above is due to Simpson \cite{Sim85}.
He also show that one can use Borel functions $f$ in the definition and obtain the same notion.

\begin{example}
For a natural number $k$, the discrete order $\mathcal{Q}=(k;=)$, denoted by $k$, is a bqo.
More generally, every finite partial ordering is a bqo.
\end{example}

Every bqo is also a {\em well-quasi-order} (often abbreviated as wqo), that is, that it is well-founded and has no infinite antichain.
Bqo's where introduced by Nash-Williams to prove wqo results, as bqo's have better closure properties than wqo's under infinitary operations.
For instance, Laver \cite{Lav78} showed that if $\Q$ is a bqo, then so are $\tree(\Q)$ ordered by the $\treeleq$, and the class of scattered $Q$-labeled linear orderings ordered by $\leq_Q$-preserving embeddabillty.
The most relevant such result for us is the following:

\begin{theorem}[{van Engelen--Miller--Steel \cite[Theorem 3.2]{EMS}}] \label{thm:Wadge-bqo}
If $\mathcal{Q}$ is a bqo, then the Wadge degrees of $\mathcal{Q}$-valued Borel functions on $\om^\om$ form a bqo too.
\end{theorem}

\subsection{Self-duality and join-reducibility}\label{sec:Wadge-key-facts}

Two important notions when trying to understand the notion of the $\Q$-Wadge degrees is that of $\sigma$-join-reducibility and self-duality.

\begin{definition}
We say that a $\mathcal{Q}$-Wadge degree $\mathbf{a}$ is {\em $\sigma$-join-reducible} if $\mathbf{a}$ is the least upper bound of a countable collection $(\mathbf{b}_i)_{i\in\om}$ of $\mathcal{Q}$-Wadge degrees such that $\mathbf{b}_i<_w\mathbf{a}$.
Otherwise, we say that $\mathbf{a}$ is {\em $\sigma$-join-irreducible}.
\end{definition}

\begin{definition}[Louveau and Saint-Raymond \cite{LS90}]
We say that a function $\mathcal{A}\colon\om^\om\to\mathcal{Q}$ is {\em self-dual} if there is a continuous function $\theta\colon\om^\om\to\om^\om$ such that $\mathcal{A}(\theta(X))\not\leq_\mathcal{Q}\mathcal{A}(X)$ for all $X\in\om^\om$.
\end{definition}

For example, in the case $\Q=2$, the $\Delta$ Wadge degrees are the self-dual ones, and the $\Si$'s and the $\Pi$'s are not.
Also, each $\Delta$ degree is the least upper bound of the $\Sigma$ degree and the $\Pi$ degree immediately below it.

Before stating the equivalence of these two notions, the following definition gives us a useful tool to study the Wadge degree of a function $\A\colon\om^\om\to\Q$.
For $\si\in\om^{<\om}$, define the function $\A\upto[\si]$ by $(\A\upto[\si])(X)=\A(\sigma\fr X)$ for any $X\in\om^\om$ (see also Observations \ref{obs:total-wadge} and \ref{obs:total-wadge2}), where $\sigma\fr X$ is the concatenation of $\sigma$ and $X$.
Notice that for each $\si\in \om^{<\om}$, $\A\upto[\si]\leq_w\A$ by essentially the identity operation. 
For some of these $\si$ we will have $\A\upto[\si]\equiv_w\A$ and for some $\A\upto[\si]<_w\A$.
Define
\[
\F(\A)=\{X:(\forall n)\;\A\upto[X\upto n]\equiv_w\A\}.
\]

One more definition, given $\A_n\colon\om^\om\to\mathcal{Q}$, $\bigoplus_{n\in\om}\A_n$ is defined by 
\[
(\bigoplus_{n\in\om}\A_n)(n\fr X)=\A_n(X).
\]

\begin{proposition}\label{prop:sji-deg-pres}
Let $\mathcal{Q}$ be a bqo and  $\A\colon\om^\om\to\mathcal{Q}$ a Borel function.
The following are equivalent  
\begin{enumerate}
\item $\A$ is $\sigma$-join-reducible.    	\label{part: ji}
\item $\A\equiv_w\bigoplus_{n\in\om}\A_n$, for some $\A_n$ which are $\si$-join-irreducible and $\A_n<_w\A$.   \label{part: sqsum}
\item $\F(\A)$ is empty.				\label{part: FA}
\item $\A$ is self-dual				\label{part: self-dual}
\end{enumerate}
\end{proposition}
\begin{proof}
The equivalence between (\ref{part: ji}) and (\ref{part: self-dual}) was proved by Block \cite[Proposition 3.5.4]{Blo14}, and is a generalization of Steel--van Wesep's theorem \cite{Wesep78} from $\Q=2$ to general $\Q$.

Let us prove (\ref{part: FA})$\implies$(\ref{part: ji}).
Suppose $\F(\A)$ is empty, and let $V$ be the set of minimal stings in $\om^{<\om}$ such that $\A\upto[\si]<_w\A$.
Then $\{[\si]: \si\in V\}$ is a clopen partition of $\om^\om$.
It is not hard to see that $\A\equiv_w \bigoplus_{\si\in V}\A\upto[\si]$, and hence that $\A$ is $\sigma$-join-reducible.

For the direction (\ref{part: ji})$\implies$(\ref{part: sqsum}), suppose that $\A$ is $\sigma$-join-reducible, and that its Wadge degree is the least upper bound of $\B_i$, for $i\in\om$, with $\B_i<_w\A$.
Since $\B_j\leq_w\bigoplus_{i\in\om}\B_i$ for all $j\in\om$, we get that $\A\leq_w\bigoplus_{i\in\om}\B_i$.
Furthermore, since $\mathcal{Q}$-Wadge degrees are bqo, we can use transfinite induction and assume that each $\B_i$ is either $\sigma$-join-irreducible or a sum of  $\sigma$-join-irreducibles. 
We would then get that $\A$ is itself equivalent to a sum of $\sigma$-join-irreducibles.

For (\ref{part: sqsum})$\implies$(\ref{part: FA}), let $\theta$ witness that $\A\leq_w\bigoplus_{i\in\om}\A_i$.
For each $X\in \om^\om$, there exists $n$ such that $\theta(X\upto n)$ is non-empty.
If $i$ is the first entry of $\theta(X\upto n)\mnskip$, we then get that $\theta$ witnesses that $\A\upto[X\upto n]\leq_w\A_i<_w\A$.
It follows that $X\not\in \F(\A)$ and hence that $\F(\A)$ is empty.
\end{proof}

\subsection{Conciliatory functions}\label{ss: conciliatory}

There is another way of characterizing non-self-dual functions, and it is using conciliatory functions. 
Essentially, these are functions whose domain is $\om^{\leq \om}$ instead of just $\om^\om$.
For a Borel function $\A\colon\om^\om\to\mathcal{Q}$, it will follow from our results that $\A$ is non-self-dual if and only if it can be extended to a function $\hat{\A}\colon\om^{\leq \om}\to\mathcal{Q}$ that is Wadge equivalent to $\A$ (as defined below).
This was proved by Duparc \cite{Dup01} for $\Q=2$ --- he actually introduced the notion of a conciliatory set.
We generalize the notion of a conciliatory set in the $\Q$-valued setting and prove this result as a consequence of Proposition \ref{prop:sublemma1(2)->(1)} and Observation \ref{obs:tree-is-conciliatory}.

To be able to deal with Wadge reducibility and with complexity pointclasses, we will use the following representation of conciliatory functions. 
Fix a symbol `${\sf pass}$' and define 
\[
\hat{\om}=\om\cup\{{\sf pass}\}.
\]
Given $X\in\hat{\om}^\om$, we use the notation $X\mnskip\in\om^{\leq\om}$ to denote the string obtained by removing all ${\sf pass}$'s from $X$ (see also the definition of the Wadge game; Section \ref{sec:Wadge-game}).

\begin{definition}
A function $\A\colon \hat{\om}^\om\to\Q$ is {\em conciliatory} if
\[(\forall X,Y\in{\hat{\om}}^\om)\;[X\mnskip=Y\mnskip\;\Longrightarrow\;\A(X)=\A(Y)].\]
A function $\Psi\colon \hat{\om}^\om\to\hat{\om}^\om$ is {\em conciliatory} if
\[(\forall X,Y\in{\hat{\om}}^\om)\;[X\mnskip=Y\mnskip\;\Longrightarrow\;\Psi(X)\mnskip=\Psi(Y)\mnskip].\]
\end{definition}

Conciliatory functions are in one-to-one correspondence with functions $\om^{\leq\om}\to\Q$ and  $\om^{\leq\om}\to\om^{\leq\om}$ respectively.
However, when we think of their Wadge degrees and of their complexity, it is better to think of them as maps defined on $\hat{\om}^\om$.
The obvious topology to give to $\hat{\om}^\om$ is the product topology of the discrete space $\hat{\om}$, which is homeomorphic to $\om^\om$ (just because there is a bijection between $\hat{\om}$ and $\om$).
We will thus treat $\hat{\om}^\om$ exactly as we treat $\om^\om$ when we define complexity classes of sets and functions. 
For instance, a Wadge reduction between conciliatory functions $\A\colon \hat{\om}^\om\to\Q$ and $\B\colon \hat{\om}^\om\to\Q$, would be continuous function $\theta\colon \hat{\om}^\om\to\hat{\om}^\om$ which is not necessarily conciliatory.
Thus, this function $\theta$ is not necessarily well-defined as a function on $\om^{\leq\om}$.

Via the identification between $\hat{\om}^\om$ and $\om^\om$, conciliatory function are just a special class of regular functions. 
Then, for instance, we can then transform a conciliatory function $\A\colon \hat{\om}^\om\to\Q$ into a function $\underline{\A}\colon \om^\om\to\Q$ which is Wadge equivalent to $\A$.
Thus, the conciliatory Wadge degrees are just a subset of the standard Wadge degrees of functions on $\om^\om$.
However, they will be very useful to us when we define the $\Sigma_T$-complete functions $\Omega_T$.

\begin{observation}\label{obs:skip-legal}
Every conciliatory function is $\sigma$-join-irreducible.
\end{observation}
\begin{proof}
If $\A$ is conciliatory, it is easy to see that ${\sf pass}^\om\in\F(\A)$, where ${\sf pass}^\om$ is the infinite sequence consisting only of ${\sf pass}$.
Thus, by Proposition \ref{prop:sji-deg-pres}, $\A$ is $\sigma$-join-irreducible.
\end{proof}

It is the converse direction of this observation that is hard to prove.

The following lemmas and observations will help us get gain some intuition on conciliatory functions, even though they will not be used in the rest of the paper.

\begin{obs}
Every partial computable operator $\Phi_e$ can be viewed as a conciliatory function.
Essentially, it just outputs {\pass}es while it is waiting either for a new value of the oracle, or a new computation to converge. 
By the same reason, every continuous function $\om^\om\to\om^\om$ can be extended to a conciliatory function as we mentioned at the end of Section \ref{sec:Q-wadge}.
\end{obs}

\begin{lemma}
A function $G\colon\om^{\leq\om}\to\om^{\leq\om}$ can be represented as a continuous conciliatory function $\hat{\om}^\om\to\hat{\om}^\om$ if and only if $\si\subseteq \tau$ implies $G(\si)\subseteq G(\tau)$ for every $\si,\tau\in\om^{<\om}$, and $G(X)=\bigcup_n G(X\upto n)$ for every $X\in\om^\om$.
\end{lemma}
\begin{proof}[Sketch of the proof]
For the left-to-right implication, suppose $\hat{G}$ is continuous conciliatory function such that $\hat{G}(X)\mnskip=G(X\mnskip)$ for all $X\in \hat{\om}^\om$.
Suppose $\tau=\si\concat\gamma$.
Every initial segment of $G(\si)$ must be an initial segment of $G(\tau)$ because every initial segment of $G(\si)\mnskip$ is contained in $\hat{G}(\si\concat\pass\ \pass\cdots \pass)\mnskip$ for some number of passes.
Then, 
\[
G(\tau)= \hat{G}(\si\concat\pass\ \pass\cdots \pass\concat \gamma\concat \pass^\om)\mnskip.
\]
It follows that $G(\si)\subseteq G(\tau)$.
By the same argument, if $\si\subseteq X$, then $G(\si)\subseteq G(X)$.
We leave the remaining details to the reader. 
\end{proof}

One can show that a function $G\colon\om^{\leq\om}\to \Q$ can be represented as a continuous conciliatory function $\hat{G}:\hat{\om}^\om\to\Q$ if and only if it is constant. 
(Just think of $Q$ as $\om$, being the first entry of the output of a function as in the lemma.)
The case of $\Si^0_2$ function gets more interesting. 

\begin{lemma}
A function $G\colon\om^{\leq\om}\to\om^{\leq\om}$ can be represented as a ${\bf\Sigma}^0_2$ conciliatory function if and only if for every $X\in \om^\om$, $G(X)$ is the pointwise limit of $G(X\upto m)$ in the following sense: for every $\si\in \om^{<\om}$, 
\[
\sigma\subseteq G(X)\iff \exists n\forall m>n\ (\sigma\subseteq G(X\upto m)).
\]
\end{lemma}

In particular, a function $G\colon\om^{\leq\om}\to\Q$ is $\bf\Si^0_2$ conciliatory if and only if $G(X)=\lim_n G(X\upto n)$ for every $X\in\om^\om$. 

\begin{proof}[Sketch of the proof]
For the left-to-right implication, suppose $\hat{G}$ is $\bf\Si^0_2$ conciliatory function such that $\hat{G}(X)\mnskip=G(X\mnskip)$ for all $X\in \hat{\om}^\om$.
By definition, the predicate $\tau\subseteq\hat{G}(X)$ is $\Sigma^0_2$-definable with parameters.
For $\sigma\in\om^{<\om}$ and $X\in\om^\om$, note that the predicate $\sigma\subset\hat{G}(X)\mnskip$ is equivalent to the existence of $\tau\in\hat{\om}^{<\om}$ such that $\tau\mnskip=\sigma$ and $\tau\subseteq\hat{G}(X)$.
The latter condition is also $\Sigma^0_2$-definable with parameters.
Thus, there is an $R$ such that 
\[
\si\subseteq \hat{G}(X)\mnskip\iff \exists n\forall m>n \ R(\si, n,m,X\upto m)
\]
for $\si\in \om^{<\om}$ and $X\in \hat{\om}^\om$.
Suppose, toward a contradiction that $\si\subseteq G(X)$, but there exits $k_0<k_1<\cdots$ such that $\si\not\subseteq G(X\upto k_n)$.
We will then define $Y\in\hat{\om}^\om$ with $Y\mnskip=X\mnskip$ such that $\si\not\subseteq G(Y)=G(X)$.
We define $Y$ by finite approximations $Y_0\subseteq Y_1\subseteq Y_2\cdots$ so that $Y_n\mnskip = X\upto k_n\mnskip$.
At each stage $n$, since $\si\not\subseteq G(X\upto k_n)$, there is an $m_n$ such that $\neg R(\si, n,m_n,Y_n\fr {\sf pass}^k)$, where $k$ is so that $|Y_n\fr {\sf pass}^k|=m_n$. 
Define $Y_{n+1}$ to be $Y_n\fr {\sf pass}^k\fr X\upto [k_{n}+1,k_{n+1}]$, so that $Y_{n+1}\mnskip = X\upto k_{n+1}\mnskip$.
We then have $\forall n\neg R(\si, n,m_n,Y\upto m_n)$, and hence that $\si\not\subseteq G(Y)$.

We leave the converse direction to the reader.
It is a standard argument in computability theory.
\end{proof}

The following lemma is also quite standard. 
It is just a uniform version of the limit lemma.

\begin{lemma}
Every partial $\bf\Si^0_2$ function $G\colon\om^\om\to\om^\om$ can be extended to a $\bf\Si^0_2$ conciliatory function $\hat{G}\colon\hat{\om}^\om\to\hat{\om}^\om$, so that $\hat{G}(X)=G(X)$ for all $X\in \om^\om$.
\end{lemma}

\subsection{Universal $\mathbf{\Sigma}^0_2$ conciliatory functions}\label{sec:conciliatory}

First, let us observe that there is no universal total $\mathbf{\Sigma}^0_2$-measurable function on $\om^\om$, as it would be $\mathbf{\Delta}^0_2$, and there is no greatest $\mathbf{\Delta}^0_2$ Wadge degree.
This is the main reason we need to deal with conciliatory functions in this paper.
Hereafter, for functions $\A,\B:\mathcal{X}\to\hat{\om}^\om$ for $\mathcal{X}\in\{\om^\om,\hat{\om}^\om\}$, we write $\A\equiv_{\sf p}\B$ if $\A(X)\mnskip=\B(X)\mnskip$ for all $X\in\mathcal{X}$.

\begin{definition}
Let $\U\colon \hat{\om}^\om\to\hat{\om}^\om$ be a conciliatory function.
We say that $\U$ is {\em $\mathbf{\Sigma}^0_2$-universal} if it is $\mathbf{\Sigma}^0_2$-measurable, if for every $\mathbf{\Sigma}^0_2$-measurable conciliatory function $G\colon \hat{\om}^\om\to\hat{\om}^\om$, there exits a continuous function $\theta\colon\hat{\om}^\om\to\hat{\om}^\om$ such that $G\equiv_{\sf p}\U\circ\theta$.
\end{definition}

Let us define a $\mathbf{\Sigma}^0_2$-universal function $\U$.
Let $\{\si_n:n\in\om\}$ be an effective enumeration of $\om^{<\om}$.
Think of an input $Y$ to $\U$ as a code for a sequence of strings $\si_{Y(0)},\si_{Y(1)},\si_{Y(2)},...$ and $\U(Y)$ as the pointwise limit of these strings. 
That is, we would like to define $\U(Y)(j)=\lim_{i\to\infty}\si_{Y(i)}(j)$ if the limit exists, and let it be undefined otherwise, except that be have to be a bit careful to get $\U$ to be of the right form.
The actual definition is as follows. For $\si\in \om^{<\om}$, $\si\neq\emptyset$,
\[
\si\subseteq \U(Y)
	\iff
\exists n\left( \si\subseteq \si_{Y\mnskip(n)} \and \forall m>n \left(Y\mnskip(m)\downarrow\ \rightarrow \sigma\subseteq \sigma_{Y\mnskip(m)}\right)\right),
\]
where $Y\mnskip(m)\downarrow$ means that $|Y\mnskip|>m$.
It is not hard to see that if $\tau_0\subseteq \U(Y)$ and $\tau_1\subseteq \U(Y)$, then $\tau_0$ and $\tau_1$ must be compatible. 
We let $\U(Y)$ be the union of all $\si$ such that $\si\subseteq \U(Y)$.
We let the reader verify that $\U$ is a $\mathbf{\Sigma}^0_2$-universal conciliatory function, as it is a standard computability theoretic argument.

$\U$ has a particular property that will be quite important: the value of $\U(Y)$ does not depend on initial segments of $Y$, and only depends on the tail of $Y$.

\begin{definition}\label{def: initiazable}
A function $\A\colon\hat{\om}^{\om}\to \hat{\om}^{\om}$ is  {\em initializable} if for every $\tau\in \hat{\om}^{<\om}$, there is a continuous function $\theta_\tau\colon \hat{\om}^\om\to[\tau]$ such that $\A\equiv_{\sf p}\A\circ\theta_\tau$.
\end{definition}

To see that our function $\U$ is initializable, suppose $0$ is the code for the empty string (i.e., $\si_0=\emptyset$), then let $\theta_\tau(Y)= \tau\fr 0\fr Y$.
It is not hard to see that $\U(Y)=\U(\tau\fr 0\fr Y)$.

We have proved the following proposition.

\begin{proposition}\label{prop:universal}
There is a $\mathbf{\Sigma}^0_2$-universal initializable conciliatory function.
\end{proposition}


\section{Nested labeled trees}\label{section:qo-nestedtrees}

In this section we give formal definitions of  $\forest^n(\Q)$, $\treeleq$, $\Sigma_T$, and the $\Sigma_T$-complete function $\Omega_T$.
We end the section by extending these ideas to all infinite Borel ranks.

\subsection{Nested Trees}\label{sec:finite-Borel-ranks-treeorder}

Let us first give some intuition for the connection between nested labeled trees and Borel functions.
First consider the characteristic function $\chi_U$ of an open set $U\subseteq\om^\om$.
Since the predicate $x\in U$ can be described by an existential formula, we have an approximation procedure which starts by guessing $\chi_U(x)=0$ until $x\in U$ is witnessed, and then changes the guess to $\chi_U(x)=1$ after seeing such a witness.
We denote the collection of all such guessing procedures, namely the pointclass $\mathbf{\Sigma}^0_1$, by the term $\langle 0\rangle\mindchange \langle 1\rangle$.
We think of the term $\langle 0\rangle\mindchange \langle 1\rangle$ as representing a tree with two nodes whose root is labeled by $0$, and leaf is labeled by $1$.
Similarly, we use the tree $\langle 1\rangle\mindchange\langle 0\rangle$ (with a root note labeled $1$, and a leaf node labeled $0$) to name the pointclass $\mathbf{\Pi}^0_1$, and we use trees of the form $\langle 0\rangle\mindchange\langle 1\rangle\mindchange\dots\mindchange\langle 0\rangle \mindchange\langle 1\rangle$ to name the finite levels of the Hausdorff-Kuratowski difference hierarchy.

To represent self-dual pointclasses such as $\mathbf{\Delta}^0_1$, we will need to consider forests rather than trees.
Given a clopen set $C\subseteq\om^\om$, one decides whether $\chi_C(x)=0$ or $\chi_C(x)=1$ at once and there is no change of mind afterwords. 
We represent this procedure by the term $\langle 0\rangle\sqcup\langle 1\rangle$, which is identified with a forest consisting of two roots labeled by $0$ and $1$, respectively.
All levels of the Hausdorff-Kuratowski difference hierarchy (hence all Wadge degrees of $\mathbf{\Delta}^0_2$ subsets of $\om^\om$) are named by terms obtained from the operations $\mindchange$ and $\sqcup$ (that is, well-founded $\{0,1\}$-labeled trees and their disjoint unions).
For instance, a term of the form $I_n \mindchange\bigsqcup_k I_k$, (where $I_\ell$ is the chain of the form $\langle 0\rangle\mindchange\langle 1\rangle\mindchange\dots\mindchange\langle 0\rangle \mindchange\langle 1\rangle$ of length $\ell$) names the $(\om+n)^{th}$-level of the difference hierarchy. 
To represent $\Delta^0_2$ 3-partitions, Selivanov used forests labeled with $\{0,1,2\}$ instead.
The idea is the same: A $\{0,1,2\}$-labeled tree guides the mind changes allowed when defining a $\Delta^0_2$ 3-partitions; since the tree is well-founded, the guessing process eventually stops.

If we want to move on to $\Delta^0_3$ functions, that is when we need to start nesting trees.
For instance, the tree $\langle T\rangle$ consisting only of a root labeled by a tree $T$, is thought of as the {\em jump} of the pointclass named by $T$.
Thus, $\langle \langle 0\rangle\mindchange\langle 1\rangle\rangle$ is the jump of $\mathbf{\Sigma}^0_1$ --- namely $\mathbf{\Sigma}^0_2$.
By using nesting of trees in this way, we will be able to climb up the Borel hierarchy.

\subsubsection{Language and terms}\label{sec:lang-terms}

All $\Q$-valued Borel functions of {\em finite} rank will be described using terms (identified with forests) in the language consisting of constant symbols (corresponding to elements in $\Q$), and three function symbols: $\mindchange$ (concatenation), $\sqcup$ (disjoint union), and $\langle\cdot\rangle$ (labeling).
To represent $\Q$-valued Borel functions of {\em infinite} rank, we will need to add symbols representing transfinite jump operations $\langle\cdot\rangle^{\om^\alpha}$.

We formally describe the collections $\tree(\Q)$ and $\forest(\Q)$ of countable well-founded $\Q$-labeled trees and their countable disjoint unions (i.e., forests) in the following inductive manner:
\begin{enumerate}
\item If $T\in\tree(\Q)$, then $T\in\forest(\Q)$.
\item For each $q\in\Q$, the term $\langle q\rangle$ is in $\tree(\Q)$. It represents the tree with only one node labeled $q$. 
\item For any countable collection $\{T_i\}_{i\in\om}$ in $\tree(\Q)$, the term $\sqcup_iT_i$ is in $\forest(\Q)$.
Terms of the form $\sqcup_iT_i$ will be called {\em $\sqcup$-type terms}, and represent forests obtained as the disjoint union of trees $T_i$.
\item For any $q\in\Q$ and $\sqcup$-type term $T\in\forest(\Q)$, the term $\langle q\rangle\mindchange T$ is in $\tree(\Q)$.
It represents the tree obtained by joining to a root labeled $q$ all the components of the forest $T$.
\end{enumerate}
Note that $\tree(\Q)$ consist of the non-$\sqcup$-type terms in $\forest(\Q)$.
Then, define $\tree^0(\Q)=\Q$, $\tree^{n+1}(\Q)=\tree(\tree^n(\Q))$, and $\forest^{n+1}(\Q)=\forest(\tree^n(\Q))$.

The way they are defined,  $\tree^m(\Q)$ and $\tree^n(\Q)$ are disjoint whenever $m<n$.
However, we will later see that every tree is  $\tree^m(\Q)$ is equivalent to one in $\tree^n(\Q)$ (Observation \ref{obs:embeddingtreeleq}).

\subsubsection{Quasi-ordering nested trees}\label{sec:def-quasi-order-nested-trees}

In this section, we introduce a quasi-order $\treeleq$ on $\forest^{<\om}(\Q)$, which we will show is isomorphic to the Wadge quasi-ordering of $\Q$-valued functions of finite Borel rank.
To simplify our notation, we always identify $\langle T\rangle$ with $\langle T\rangle\mindchange\sqcup_i\mathbf{O}$, where $\mathbf{O}$ is the empty forest, which we think of as an imaginary least element with respect to the quasi-order $\treeleq$, that is, $\mathbf{O}\treeleq T$ for any $T\in\forest^n(\Q)$.

\begin{definition} \label{def: treeleq}
We inductively define a quasi-order $\treeleq$ on $\bigcup_n\tree^{n}(\Q)$ as follows, where the symbols $p$ and $q$ range over $\Q$, and $U$, $V$, $S$, and $T$ range over $\bigcup_n\tree^{n}(\Q)$:
\begin{align*}
p\treeleq q &\iff p\leq_\Q q,\\
\langle U\rangle\treeleq \langle V\rangle & \iff U\treeleq V,
\end{align*}
and if $S$ and $T$ are of the form $\langle U\rangle\mindchange\sqcup_iS_i$ and $\langle V\rangle\mindchange\sqcup_jT_j$, respectively, then
\[
S\treeleq T \iff
\begin{cases}
\mbox{either }  U\treeleq  V   &\mbox{and } (\forall i)\;S_i\treeleq T,\\
\mbox{or }\hspace{7mm}  U\not\treeleq  V   &\mbox{and } (\exists j)\;S\treeleq T_j.
\end{cases}
\]
This pre-ordering induces an equivalence as usual: let $S\equiv T$ if $S\treeleq T$ and $T\treeleq S$.
For $p\in\Q$, we let $p\equiv \langle p\rangle \equiv \langle\langle p\rangle\rangle\equiv\cdots$, allowing us to compare trees of different levels. 

Finally, $\treeleq$ is uniquely extended to a quasi-order on $\bigcup_n\forest^n(\Q)$ by interpreting $\sqcup$ as a countable supremum operation:
\[
\sqcup_iS_i\treeleq \sqcup_jT_j \iff (\forall i)(\exists j)\ S_i\treeleq T_j.
\]
\end{definition}

\begin{obs}\label{obs:embeddingtreeleq}
For every $T\in\forest^{\leq n}(\Q)$, there is $S\in\forest^n(\Q)$ such that $S\equiv T$.
\end{obs}

\begin{proof}
Assume that $m\leq n$ and $T\in\forest^m(\Q)$.
Then, consider the term $\iota(T)=T[\langle q\rangle^{n-m}/q]_{q\in\Q}$ obtained by substituting all occurrences of $q\in\Q$ by $\langle q\rangle^{n-m}$, where $\langle q\rangle^0=q$ and $\langle q\rangle^{k+1}=\langle \langle q\rangle^k\rangle$.
Note that $\iota(T)\in\forest^n(\Q)$, and it is clear that $T\equiv\iota(T)$.
\end{proof}

\begin{obs}
Consider $S,T\in\forest^n(\Q)$, and use $\subseteq$ to denote the ordering among the nodes of $S$ and $T$, the roots being the $\subseteq$-least elements. 
It is not hard to see that $S\treeleq T$ if and only if there exists a map $f\colon S\to T$ which is order preserving, in the sense for $\si,\tau$ nodes in $S$, $\si\subseteq\tau \implies f(\si)\subseteq f(\tau)$, and $\leq_\Q$-increasing in the sense that the label$(\si) \treeleq$ label$(f(\si))$ for every $\si\in S$.
Such $f$ does not need to be one-to-one.
\end{obs}

\begin{theorem}[Laver \cite{Lav78}]
For $n\in \om$, if $Q$ is better-quasi-ordered, then so is $\forest^{\leq n}(\Q)$.
\end{theorem}
Laver showed that if $Q$ is a bqo, so is $\tree(\Q)$ for an even stronger notion of reducibility.

\subsection{The associated pointclasses}\label{sec:natural-pointclass}

As we mentioned before, each forest $T\in\forest^n(\Q)$  defines a pointclass $\Sigma_T$.
For instance, if $\Q=2$, then 
\[
\Sigma_{\langle 0\rangle\sqcup\langle 1\rangle}=\mathbf{\Delta}^0_1,
\quad\Sigma_{\langle 0\rangle\mindchange\langle 1\rangle}=\mathbf{\Sigma}^0_1,
\quad\Sigma_{\langle 1\rangle\mindchange\langle 0\rangle}=\mathbf{\Pi}^0_1,
\quad\Sigma_{\langle\langle 0\rangle\mindchange\langle 1\rangle\rangle}=\mathbf{\Sigma}^0_2, 
\quad \mbox{ and so on.}
\]

The following observations will simplify our definitions.

\begin{observation}\label{obs:total-wadge}
Let $\F$ be a nonempty closed subset of $\om^\om$.
Then, for every function $\A\colon \F\to\mathcal{Q}$ there is a function $\widehat{\A}\colon \om^\om\to\mathcal{Q}$ which is Wadge equivalent to $\A$. 
\end{observation}

\begin{proof}
By zero-dimensionality of $\om^\om$, there is a retraction $\rho_\F\colon \om^\om\to\F$ (that is, $\rho_\F$ is continuous and $\rho_\F\upto\F$ is identity).
Define $\widehat{\A}=\A\circ\rho_\F$.
Then, we have $\widehat{\A}\leq_w\A$ via $\rho_\F$, and $\A\leq_w\widehat{A}$ via the identity map.

The definition of this retraction is quite standard:
Let $T\subseteq \om^{<\om}$ be a tree without dead end such that $\F=[T]$.
We define $\rho_\F\colon\om^{<\om}\to T$ by induction:
$\rho_\F(\si\fr n)=\rho_\F(\si)\fr m$ where $m\in\om$ is the closest to $n$ such that $\rho_\F(\si)\fr m\in T$.
(By closest we mean such that $|m+\frac{1}{3}-n|$ is least, for instance.)
We then extend $\rho_\F$ to $\om^\om$ to $\F$ by continuity. 
\end{proof}

\begin{obs}\label{obs:total-wadge2}
Let $\V$ be a nonempty open subset of $\om^\om$.
Then, for every function $\A\colon \V\to\mathcal{Q}$ there is a function $\widehat{\A}\colon \om^\om\to\mathcal{Q}$ which is Wadge equivalent to $\A$. 
\end{obs}

\begin{proof}
Let $V=\{\tau_0,\tau_1,...\} \subseteq\om^{<\om}$ be a generator of $\V$.
That is, $V$ is so that $\{[\tau]: \tau\in V\}$ is a partition on $\V$ in clopen sets.
Then the bijection $n\fr X \mapsto \tau_n\fr X\colon \om^\om \to \V$ induces a function $\widehat{\A}\colon \om^\om\to\mathcal{Q}$ Wadge equivalent to $\A$
\end{proof}

From now one, whenever we encounter a $\Q$-valued function whose domain is an either open or closed subsets of $\om^\om$, we identify it with the corresponding function of domain $\om^\om$.

\begin{definition}\label{def: Sigma T}
For each $T\in\bigcup_n\forest^n(\Q)$, we inductively define the class $\PSigma_T$ of $\Q$-valued functions on $\om^\om$ as follows:
\begin{enumerate}
\item $\PSigma_q$ consists only of the constant function $X\mapsto q\colon \om^\om\to\Q$.
\item If $T$ is of the form $\sqcup_iS_i$, then $\A\in\PSigma_T$ if and only if there are clopen partition $(\C_i)_{i\in\om}$ of $\om^\om$ such that $\A\upto\C_i\in\PSigma_{S_i}$ for each $i\in\om$.
\item $\A\in\PSigma_{T\mindchange S}$ if and only if there is an open set $\V\subseteq\om^\om$ such that $\A\upto(\om^\om\setminus\V)$ is in $\PSigma_T$ and $\A\upto\V$ is in $\PSigma_S$.
\item $\A\in\PSigma_{\langle T\rangle}$ if and only if there is a $\mathbf{\Sigma}^0_2$-measurable function $\D\colon \om^\om\to\om^\om$ and a $\PSigma_T$-function $\B\colon \om^\om\to\Q$ such that $\A=\B\circ\D$.
\end{enumerate}

We say that a function $\A\colon \om^\om\to\Q$ is {\em $\PSigma_T$-complete} if $\A\in\PSigma_T$ and every $\PSigma_T$-function $\B$ is Wadge reducible to $\A$.
\end{definition}

\begin{obs}\label{obs:sigma-continuous}
Let $T\in\forest^n(\Q)$ be a forest, and $\theta\colon \om^\om\to\om^\om$ be a continuous function.
If $\A\colon \om^\om\to\Q$ is in $\PSigma_T$, then so is $\A\circ\theta$.
This can be easily shown by induction on $T$ as a term.
\end{obs}


To prove Proposition \ref{prop: measurability of Sigma T} for functions of finite Borel rank, we check measurability of $\PSigma_T$-functions.
We denote by $\mathbf{\Delta}^0_\xi$ the set of all $\mathbf{\Delta}^0_\xi$-measurable functions.

\begin{lemma}\label{lem:Borel-rank-calc}
Let $S,T$ be terms and $\xi$ be a countable ordinal.
\begin{enumerate}
\item $\PSigma_T,\PSigma_S\subseteq\mathbf{\Delta}^0_{2+\xi}$ implies $\PSigma_{T\mindchange S}\subseteq\mathbf{\Delta}^0_{2+\xi}$.
\item $\PSigma_T\subseteq\mathbf{\Delta}^0_{1+\xi}$ implies $\PSigma_{\langle T\rangle}\subseteq\mathbf{\Delta}^0_{2+\xi}$.
\end{enumerate}
\end{lemma}

\begin{proof}
To see (1), let $\A$ be a $\PSigma_{T\mindchange S}$-function.
Then, there is an open set $\V$ such that $\A\upto(\om^\om\setminus\V)$ can be extended to a total $\Sigma_T$-function $\A_0$, and $\A\upto\V$ can be extended to a total $\Sigma_S$-function $\A_1$ as in Observations \ref{obs:total-wadge} and \ref{obs:total-wadge2}.
Thus, for any $q\in\Q$, we have
\begin{eqnarray*}
\A^{-1}[q]&=&(\A_0^{-1}[q]\setminus\V)\cup(\A_1^{-1}[q]\cap\mathcal{V}) \\
\A^{-1}[q]^c&=&(\A_0^{-1}[q]^c\setminus\V)\cup(\A_1^{-1}[q]^c\cap\mathcal{V}) 
\end{eqnarray*}
This gives $\mathbf{\Sigma}^0_{2+\xi}$-definitions of $\A^{-1}[q]$ and $\A^{-1}[q]^c$ since $\A_0$ and $\A_1$ are $\mathbf{\Delta}^0_{2+\xi}$-measurable, and $\V$ is open.
Consequently, $\A$ is $\mathbf{\Delta}^0_{2+\xi}$-measurable.

For (2), note that the composition $\B\circ\mathcal{D}$ of a $\mathbf{\Delta}^0_{1+\xi}$-measurable function $\B$ and a $\mathbf{\Sigma}^0_{1+\eta}$-measurable function $\mathcal{D}$ is always $\mathbf{\Delta}^0_{1+\eta+\xi}$-measurable.
This is because, one can see that if $\mathcal{S}\subseteq\om^\om$ is $\mathbf{\Sigma}^0_{1+\zeta}$ then $\mathcal{D}^{-1}[\mathcal{S}]$ is $\mathbf{\Sigma}^0_{1+\zeta+\xi}$ by induction on Borel rank.
Now, every $\A\in\PSigma_{\langle T\rangle}$ is given by the composition of the $\Sigma_T$-function $\B$ and the $\mathbf{\Sigma}^0_{2}$-measurable function $\D$.
This implies that $\B\circ\D$ is $\mathbf{\Delta}^0_{2+\xi}$-measurable since $\B$ is $\mathbf{\Delta}^0_{1+\xi}$-measurable by our assumption.
\end{proof}

In particular, if $T\in\forest^n(\Q)$, then every $\PSigma_T$-function is $\mathbf{\Delta}^0_{n+1}$-measurable.
As a consequence, this verifies Proposition \ref{prop: measurability of Sigma T} for functions of finite Borel rank.

\subsection{$\Sigma_T$-complete functions}\label{sec:universal-function}

For $\Q=2$, Duparc \cite{Dup01} defined a complete sets $\Omega_\nu\subseteq\om^{\leq\om}$ for the different levels of the 2-Wadge hierarchy.
For $Q=k\in\om$, Selivanov \cite{Sel07} defined complete $\Delta^0_2$ functions $\mu_T\colon \om^\om\to k$ for each forest $T\in\forest(k)$ based on the similar ideas.
In this section we extend Duparc and Selivanov's definition to all bqos $\mathcal{Q}$ and nested forests $T\in\forest^n(\mathcal{Q})$, and later on throughout the Borel hierarchy.

The complete functions we will define are conciliatory; see Section \ref{ss: conciliatory}.

\subsubsection{Difference hierarchy and mind-change operation}
The Hausdorff-Kuratowski difference hierarchy (and the Ershov hierarchy) can be understood using the notion of a {\em mind-change}.
That is, a subset $\A$ of $\om^\om$ is in the $n^{th}$-level of the difference hierarchy if and only if the characteristic function of $\A$ is approximated by a continuous function with $n$ mind-changes.

We want to define an operation $\A\mindchange\B$ for $\A,\B\colon \hat{\om}^{\om}\to\Q$ which represents a function that could act as $\A$, but at any time could change its mind and act as $\B$.
To make it easier to describe such a process, we introduce notations representing this kind of approximation procedure.
Suppose we first want to output a sequence and after $\ell$ steps, after having defined a sequence $Y\in\om^{\ell}$, we change our mind and we want to output a new sequence $Z\in\om^{\leq\om}$
We will encode this by the following real:
\[
Y\mindchange Z:=\langle 2Y(0),2Y(1),\dots,2Y(\ell-1),2Z(0)+1,Z(1),Z(2),\dots\rangle.
\]
We want to define this procedure on $\hat{\om}^\om$ as follows, where we will require that the first entry of the second sequence $Z$ is not {\sf pass}, that is, $Z\in\om\times\hat{\om}^\om$.

\begin{notation}
Given $Y\in\hat{\om}^{\ell}$ of length $\ell\in\om\cup\{\om\}$ and $Z\in\om\times\hat{\om}^{\om}$, we define $Y\mindchange Z$ as folllows:
\[
Y\mindchange Z(n)=
\begin{cases}
2Y(n),&\mbox{ if }n<\ell\mbox{ and }Y(n)\not={\sf pass},\\
{\sf pass},&\mbox{ if }n<\ell\mbox{ and }Y(n)={\sf pass},\\
2Z(0)+1,&\mbox{ if }n=\ell,\\
Z(n-\ell+1),&\mbox{ if }n>\ell.
\end{cases}
\]
\end{notation}

\begin{obs}\label{obs:change-bijec}
The map $(Y,Z)\mapsto Y\mindchange Z$ admits a conciliatory inverse.
Indeed, there uniquely exist conciliatory continuous functions $\pi_0$ and $\pi_1$ such that for any $X\in\hat{\om}^\om$, 
\[
X=\pi_0(X)\mindchange\pi_1(X).
\]
\end{obs}


We hereafter fix such functions $\pi_0,\pi_1$.
Note that $\pi_1(X)\mnskip$ is nonempty if and only if $X$ has changed his mind at some point.
In other words, if $\pi_1(X)\mnskip$ is empty, then the sequence given by $X$ is $\pi_0(X)$, and if $\pi_1(X)\mnskip$ is nonempty, $X$ has already deleted the former sequence $\pi_0(X)$, and now proposes $\pi_1(X)$.

\begin{notation}[see also {\cite[Definition 6]{Dup01}} for $\Q=2$]
Let $\A$ and $\B$ be functions whose domains are subsets of $\hat{\om}^{\om}$.
We define a function $\A\mindchange\B\colon \hat{\om}^{\om}\to\Q$ as follows:
\[(\A\mindchange\B)(X)=
\begin{cases}
\A(\pi_0(X))&\mbox{ if }\pi_1(X)\mbox{ is empty},\\
\B(\pi_1(X))&\mbox{ otherwise.}
\end{cases}
\]
\end{notation}

It is easy to check that the operation $\to$ can also be seen as an operation on the Wadge degrees:

\begin{observation}\label{obs:mind-change}
Let $\A,\B,\C,\D$ be functions whose domains are subsets of $\hat{\om}^{\om}$.
If $\A\leq_w\C$ and $\B\leq_w\D$ then $\A\mindchange\B\leq_w\C\mindchange\D$.
\end{observation}

\subsubsection{$\Sigma_T$-complete functions}

We now inductively assign a function $\Omega_T$ to each forest $T\in\forest^n(\mathcal{Q})$, and we will show that $\Omega_T$ is $\Sigma_T$-complete.
Recall that $T$ is a tree if and only if the outermost function symbol is not the disjoint union $\sqcup$, and thus, $\Omega_T$ is defined by the construction in (1), (3), or (4) of Definition \ref{def: Sigma T}.
If $T$ is a tree, $\Omega_T$ will be a conciliatory function from $\hat{\om}^{\om}$ to $\mathcal{Q}$.
If $T$ is not a tree, $\Omega_T$ will be a function from $\om\times\hat{\om}^{\om}$ to $\mathcal{Q}$, which is almost conciliatory, that is, $X\mnskip=Y\mnskip$ implies $\Omega_T(n\fr X)=\Omega_T(n\fr Y)$ for any $n\in\om$.
(Think of almost conciliatory functions as having domain $\om^{\leq\om}\setminus\{\emptyset\}$.)

\begin{definition}\label{def:universal-function}
Let $T\in\forest^n(\mathcal{Q})$.
We inductively define $\Omega_T$ as follows:
\begin{enumerate}
\item Suppose that $T$ is of the form $\langle q\rangle$ for some $q\in\mathcal{Q}$.
Then define $\Omega_{\langle q\rangle}\colon \hat{\om}^{\om}\to\mathcal{Q}$ as the constant function $X\mapsto q$, that is,
\[(\forall X\in\hat{\om}^{\om})\ \ \ \Omega_{\langle q\rangle}(X)=q.\]
We sometimes abbreviate $\Omega_{\langle q\rangle}$ to $\Omega_q$.
\item Suppose that $T$ is of the form $\bigsqcup_nT_n$, where each $T_n$ is a tree.
Then define $\Omega_T\colon \om\times\hat{\om}^{\om}\to\mathcal{Q}$ as follows:
\[
\Omega_{\bigsqcup_nT_n}(X)=\bigoplus_{n\in\om}\Omega_{T_n}(X)
\]
\item Suppose that $T$ is of the form $\langle S\rangle{}^{\rightarrow}F$, where $S$ is the label on the root of $T$ (thus $S\in{\rm Tree}^{n-1}(\Q)$), and $F$ is a forest.
Then,
\[\Omega_{\langle S\rangle{}^{\rightarrow}F}=\Omega_{\langle S\rangle}\mindchange\Omega_F.\]
\item Suppose that $T$ is of the form $\langle S\rangle$ for some tree $S$.
Then define $\Omega_{T}\colon \hat{\om}^{\om}\to\mathcal{Q}$ as follows:
\[
\Omega_{\langle S\rangle}=\Omega_S\circ\U,
\]
where $\U$ is a fixed $\mathbf{\Sigma}^0_2$-universal initializable conciliatory function as in Proposition \ref{prop:universal}.
\end{enumerate}
\end{definition}

\begin{observation}\label{obs:tree-is-conciliatory}
If $T\in\tree^n(\mathcal{Q})$, $\Omega_T$ is conciliatory, and if $T$ is a $\sqcup$-type term, $\Omega_T$ is almost conciliatory.
The proof is an easy induction on the term $T$.
\end{observation}



\begin{obs}\label{obs:omega in right pointclass}
For every $T\in\forest^n(\Q)$, the function $\Omega_T$ is in $\Sigma_T$.
\end{obs}
\begin{proof}
This  is obvious if $T$ if constructed from (1), (2) or (4).
Thus, it suffices to show that $\Omega_{T\mindchange S}\in\PSigma_{T\mindchange S}$.
Recall that every $X\in{\hat{\om}}^\om$ is of the form $\pi_0(X)\mindchange\pi_1(X)$ by Observation \ref{obs:change-bijec}.
Let $\V$ be an open set consisting of all sequences $X$ such that $\pi_1(X)$ is nonempty (which indicates that $X$ has changed his mind at some point).
It is clear that $\Omega_{T\mindchange S}\upto(\om^\om\setminus\V)=\Omega_T\circ\pi_0\upto(\om^\om\setminus\V)$, and $\Omega_{T\mindchange S}\upto\V=\Omega_S\circ\pi_1\upto\V$.
By induction hypothesis and by Observation \ref{obs:sigma-continuous}, the former function is in $\PSigma_T$ and the latter function is in $\PSigma_S$.
This concludes that $\Omega_{T\mindchange S}\in\PSigma_{T\mindchange S}$.
\end{proof}

\begin{lemma}\label{lemma:omega-complete}
For every $T\in\forest^n(\Q)$, the function $\Omega_T$ is $\Sigma_T$-complete.
\end{lemma}
\begin{proof}
First assume that $T$ is of the form $\sqcup_iT_i$, and let $\A$ be a $\PSigma_T$-function.
Then there is a clopen partition $(\C_i)_{i\in\om}$ such that $\A\upto\C_i$ is in $\PSigma_{T_i}$ for any $i\in\om$.
By induction hypothesis, we have a continuous function $\theta_i$ witnessing $\A\upto\C_i\leq_w\Omega_{S_i}$ for every $i\in\om$.
Thus, to see $\A\leq_w\Omega_T$, given $X\in\om^\om$ one can computably find $i_X\in\om$ such that $X\in\C_{i_X}$, and then we have $\A(X)\leq_\Q\Omega_T({i_X}\fr\theta_{i_X}(X))$.

Next, let $\A$ be a function in $\PSigma_{T\mindchange S}$ .
Then, there is an open set $\V$ such that $\A\upto\V$ is in $\PSigma_S$ and $\A\upto(\om^\om\setminus\V)$ is in $\PSigma_T$.
Recall that the former condition means that there is a generator $V$ of $\V$ such that $\A\upto[\sigma]\in\PSigma_S$ for any $\sigma\in V$.
By induction hypothesis, we have continuous functions $\theta$ witnessing $\A\upto(\om^\om\setminus\V)\leq_w\Omega_T$, and $\gamma_\sigma$ witnessing $\A\upto[\sigma]\leq_w\Omega_S$ for every $\sigma\in V$.
To see $\A\leq_w\Omega_{T\mindchange S}$, given $X\in\om^\om$, we first follow $\theta$ until we see $X\upto s\in V$ for some $s$ (if ever).
If we see $X\upto s\in V$, then we change our mind (that is, delete the former sequence $\theta(X\upto s-1)$), and now follow $\gamma_{X\upto s}$.
Recall that, in the latter case, this process is coded as $\theta(X\upto s-1)\mindchange\gamma_{X\upto s}(X)$.
This witnesses $\A\leq_w\Omega_{T\mindchange S}$.

Let $\A$ be a $\PSigma_{\langle T\rangle}$-function.
Then, there are $\mathbf{\Delta}^0_2$-function $\D$ and a $\PSigma_T$-function $\B$ such that $\A=\B\circ\D$.
By induction hypothesis, we have a continuous function $\theta\colon\om^\om\to\hat{\om}^\om$ witnessing $\B\leq_w\Omega_T$.
Thus, $\A(X)\leq_\Q\Omega_T\circ\theta\circ\D(X)$.
Since $\theta\circ\D$ is ${\bf \Sigma}^0_2$, and $\U$ is  universal, there is  a continuous function $\Psi\colon\om^\om\to\hat{\om}^\om$ such that $\theta\circ\D\equiv_{\sf p}\U\circ \Psi(X)$.
Then, since $\Omega_T$ is conciliatory by Observation \ref{obs:tree-is-conciliatory},
\[
\A(X)=\B(\D(X)) \leq_{\Q} \Omega_T(\theta\circ\D(X))=\Omega_T(\U\circ\Psi( X))=\Omega_{\langle T\rangle}(\Psi(X)).
\]
Consequently, the continuous function $\Psi$ witnesses $\A\leq_w\Omega_{\langle T\rangle}$.
\end{proof}

\begin{lemma}\label{lem: ordering ->}
For $S,T\in \forest^n(\Q)$, if $S\treeleq T$, then $\Omega_S\leq_w \Omega_T$.
\end{lemma}
\begin{proof}
We show the assertion by induction on the definition of $\treeleq$ in  \ref{def: treeleq}.
First, it is clear that $\Omega_p\leq_w\Omega_q$ if and only if $p\treeleq q$.
Next suppose that we have $\langle U\rangle\treeleq\langle V\rangle$, which is equivalent to $U\treeleq V$ by definition.
By the induction hypothesis, we have a $\Q$-Wadge reduction $\theta\colon\hat{\om}^\om\to\hat{\om}^\om$ witnessing $\Omega_U\leq_w\Omega_V$.
By the $\bf\Sigma^0_2$ universality of $\U$, there exists a continuous $\Phi\colon \hat{\om}^\om\to\hat{\om}^\om$ such that $\theta\circ\U\equiv_{\sf p}\U\circ \Phi$.
Since $\Omega_V$ is conciliatory by Observation \ref{obs:tree-is-conciliatory}, we have that for every $X\in \hat{\om}$,
\[
\Omega_{\langle U\rangle}(X)=\Omega_U(\U(X))\leq_\Q \Omega_V(\theta(\U(X)))=\Omega_V(\U(\Phi(X)))=\Omega_{\langle V\rangle}(\Phi(X)).
\]
Thus $\Phi$ witnesses that $\Omega_{\langle U\rangle}\leq_w \Omega_{\langle V\rangle}$.

Now, consider $S=\langle U\rangle\mindchange\bigsqcup_iS_i$ and $T=\langle V\rangle\mindchange\bigsqcup_jT_j$.
First consider the case that $\langle U\rangle\treeleq\langle V\rangle$.
In this case, by the definition of $\treeleq$ under the assumption that $\langle U\rangle\treeleq\langle V\rangle$, we have that $S_i\treeleq T$ for any $i\in\om$.
By the induction hypothesis, $\Omega_{S_i}\leq_w\Omega_T$ via a continuous function $\theta_i$ for any $i$, and $\Omega_{\langle U\rangle}\leq_w\Omega_{\langle V\rangle}$ via $\tau$.
The idea to define the reduction is as follows: 
Given $X\in\om^\om$, while $X$ does not change its mind, use the reduction $\tau\colon\Omega_{\langle U\rangle}\leq_w\Omega_{\langle V\rangle}$.
If $X$ changes its mind and moves into some $S_k$, then we need to use the reduction $\theta_k\colon\Omega_{S_k}\leq_w\Omega_T$.
Since we have already taken some steps within the domain of $\langle V\rangle$, we need to use the initilizability of $\Omega_{\langle V\rangle}$ to start over with the reduction to $\Omega_T$.

More formally:
By the initializability of $\Omega_{\langle V\rangle}$ (see Definition \ref{def: initiazable}), for any $\sigma$, there is a continuous function $\eta_\sigma$ witnessing $\Omega_{\langle V\rangle}\leq_w\Omega_{\langle V\rangle}\upto[\sigma]$.
We can then extend this map to a Wadge reduction $\Omega_T\leq_w \Omega_T\upto[\sigma\mindchange\emptyset]$, where $\sigma\mindchange\emptyset$ represents the string in the domain of $\Omega_T$ for which we haven't changed our mind yet, and we are still computing $\Omega_{\langle V\rangle}$.
For each $\si$, let $\hat{\eta}_\sigma$ be such that $\Omega_T(Y) \leq_\Q \Omega_T(\sigma \mindchange \hat{\eta}_\sigma(Y))$ witnessing such reduction.
Given $X$, it is of the form $\pi_0(X)\mindchange\pi_1(X)$ by Observation \ref{obs:change-bijec}.
If $\pi_1(X)$ is empty (that is, $X$ never changes his mind), then note that $\Omega_S(X)=\Omega_{\langle U\rangle}(\pi_0(X))$.
In this case, return $\tau(\pi_0(X))\mindchange\emptyset$.
Let $X$ be a sequence such that $\pi_1(X)$ is nonempty (that is, $X$ has changed his mind at some point).
Put $k=\pi_1(X)(0)$.
Note that $\Omega_S(X)=\Omega_{\bigsqcup_iS_i}(\pi_1(X))=\Omega_{S_k}(\pi_1(X)^-)$, where $\pi_1(X)=k\fr \pi_1(X)^-$.
Then change our guess to $\hat{\eta}_{\tau(\pi_0(X))}\circ\theta_k(\pi_1(X)^-)$.
We thus get
\begin{multline*}
\Omega_S(X) = \Omega_{S_k}(\pi_1(X)^-)\leq_\Q   \\
 \Omega_T(\theta_k(\pi_1(X)^-)) \leq_\Q \Omega_T(\tau(\pi_0(X))\ \ \mindchange\ \ \hat{\eta}_{\tau(\pi_0(X))}(\theta_k(\pi_1(X)^-))).
\end{multline*}
Putting these two cases together, we have a Wadge reduction from $\Omega_{S}$ to $\Omega_T$.

We now assume that $\langle U\rangle\not\treeleq\langle V\rangle$.
In this case, by definition, $S\treeleq T$ if and only if $S\treeleq T_j$ for some $j\in\om$.
By induction hypothesis $\Omega_S\leq_w\Omega_{T_j}$ for some $j$.
Clearly, this condition implies $\Omega_S\leq_w\Omega_T$.
\end{proof}

We will prove the reverse direction, that $\Omega_T\leq_w\Omega_S\implies T\treeleq S$, in Subsection \ref{sec: proof or ordering embedding}.
We need to wait until then, because we need the jump inversion operator for the proof.


\subsection{Infinite Borel ranks}\label{sec:infinite-Borel-ranks-treeorder}

We now extend our ideas from Section \ref{sec:finite-Borel-ranks-treeorder} to infinite Borel ranks.
The reader who is only interested in Borel functions of finite rank can skip Section \ref{sec:infinite-Borel-ranks-treeorder}.

\subsubsection{Language and terms (infinite Borel rank)}

Given a set $\mathcal{Q}$, let $\mathcal{L}(\mathcal{Q})$ be the language consisting of constant symbols $q$ for each $q\in\mathcal{Q}$, an $\om$-ary function symbol $\sqcup$, a two-ary function symbol ${}^\rightarrow$, and a unary function symbol $\langle\cdot\rangle^{\om^\alpha}$ for every countable ordinal $\alpha<\om_1$.

We define $\forest^{\om^\alpha}(\mathcal{Q})$ as the set of all $\mathcal{L}(\mathcal{Q})$-terms of rank below $\om^\alpha$ as follows:

\begin{definition}[Terms of Rank below $\om^\alpha$]
We inductively define the sets $\tree^{\om^\alpha}(\mathcal{Q})$ and $\forest^{\om^\alpha}(\mathcal{Q})$ consisting of $\mathcal{L}(\mathcal{Q})$-terms as follows:
\begin{enumerate}
\item If $\beta\leq\alpha$ and $T\in\tree^{\om^\beta}(\mathcal{Q})$ then $T\in\tree^{\om^\alpha}(\mathcal{Q})$ and $T\in\forest^{\om^\alpha}(\mathcal{Q})$.
\item If $q\in\mathcal{Q}$ then $q\in\tree^{1}(\mathcal{Q})$ (where note that $\om^0=1$), and call it {\em $\langle\rangle$-type}.
\item If $\beta<\alpha$ and $T\in\tree^{\om^\alpha}(\mathcal{Q})$ then $\langle T\rangle^{\om^\beta}\in\tree^{\om^\alpha}(\mathcal{Q})$,  and call it {\em $\langle\rangle$-type}.
\item If $T_i\in\tree^{\om^\alpha}(\mathcal{Q})$ for every $i\in\om$, then $\sqcup_{i}T_i\in\forest^{\om^\alpha}(\mathcal{Q})$, and call it {\em $\sqcup$-type}.
\item For any $\langle\rangle$-type term $T\in\tree^{\om^\alpha}(\mathcal{Q})$ and $\sqcup$-type term $S\in\forest^{\om^\alpha}(\mathcal{Q})$, the term $T\mindchange S$ is in $\tree^{\om^\alpha}(\mathcal{Q})$.
\end{enumerate}
\end{definition}

We define $\tree^{\om_1}(\mathcal{Q})=\bigcup_\alpha\tree^\alpha(\mathcal{Q})$ and $\forest^{\om_1}(\mathcal{Q})=\bigcup_\alpha\forest^\alpha(\mathcal{Q})$.

For instance, $\forest^\om(\Q)$ can be viewed as the closure of $\Q$ of the operations $\langle\cdot\rangle$, $\sqcup$, and $\mindchange$.
Notice that this is far larger than $\bigcup_n\forest^n(\Q)$ (even with respect to $\treeleq$) because a term in $\forest^\om(\Q)$ can contain unbounded applications of the labeling function $\langle\cdot\rangle$, e.g., $\sqcup_n\langle 0\mindchange 1\rangle^n$.
This reflects the fact that the pointclass $\mathbf{\Delta}^0_\om$ is strictly larger than $\bigcup_{n<\om}\mathbf{\Delta}^0_n$.
On the other hand, the function $\langle\cdot\rangle^\om$ would take out of $\forest^\om(\Q)$, reflecting that the conciliatory $\Si^0_\om$ universal function is not $\mathbf{\Delta}^0_\om$.

For a set $\mathcal{R}$ of $\mathcal{L}(\mathcal{Q})$-terms and an ordinal $\alpha$, we define:
\[
\langle\mathcal{R}\rangle^{\om^\alpha}=\{\langle T\rangle^{\om^\alpha}:T\in\mathcal{R}\}.
\]

We then inductively define the set $\tree^{\om^\alpha\cdot k}(\mathcal{Q})\subseteq\tree^{\om^{\alpha+1}}(\mathcal{Q})$ of $\mathcal{L}(\mathcal{Q})$-terms (of rank $<\om^\alpha\cdot k$) as follows:
\[\tree^{\om^\alpha\cdot 1}(\mathcal{Q})=\tree^{\om^\alpha}(\mathcal{Q}),\qquad\tree^{\om^\alpha\cdot k+1}(\mathcal{Q})=\tree^{\om^\alpha}(\langle\tree^{\om^\alpha\cdot k}(\mathcal{Q})\rangle^{\om^\alpha}).\]

In general, recall that every countable ordinal $\xi$ can be written as $\om^\alpha+\beta$ for some $\beta<\om^{\alpha+1}$.
Then we define $\tree^\xi(\mathcal{Q})$ and $\forest^\xi(\mathcal{Q})$ as follows:
\[
\tree^{\om^\alpha+\beta}(\mathcal{Q})=\tree^{\om^\alpha}(\langle\tree^\beta(\Q)\rangle^{\om^\alpha}),\qquad\forest^{\om^\alpha+\beta}(\mathcal{Q})=\forest^{\om^\alpha}(\langle\tree^\beta(\Q)\rangle^{\om^\alpha}).
\]

\subsubsection{Quasi-ordering nested trees (infinite Borel rank)}

In this section, we extend the domain of the quasi-order $\treeleq$ to $\forest^{\om_1}(\Q)$.
As in Section \ref{sec:def-quasi-order-nested-trees}, we first inductively define a quasi-order $\treeleq$ on $\tree^{\om_1}(\Q)$, and then, $\treeleq$ is uniquely extended to a quasi-order on $\forest^{\om_1}(\Q)$ by interpreting $\sqcup$ as a countable supremum operation.
Recall the convention from Section \ref{sec:def-quasi-order-nested-trees} that we always identify $p\in\Q$ with $\langle p\rangle$, and $\langle T\rangle$ with $\langle T\rangle\mindchange\sqcup_i\mathbf{O}$, where $\mathbf{O}$ is the empty forest, viewed as an imaginary least element w.r.t.\ the quasi-order $\treeleq$, that is, $\mathbf{O}\treeleq T$ for any $T\in\tree^{\om_1}(\Q)$.

\begin{definition}
We inductively define a quasi-order $\treeleq$ on $\bigcup_n\tree^{\om_1}(\Q)$ as follows, where the symbols $p$ and $q$ range over $\Q$, and $U$, $V$, $S$, and $T$ range over  range over $\tree^{\om_1}(\Q)$:
\begin{align*}
p\treeleq q &\iff p\leq_\Q q,\\
\langle U\rangle^{\om^\alpha}\treeleq \langle V\rangle^{\om^\beta} & \iff
\begin{cases}
U \treeleq V & \mbox{ if }\alpha=\beta,\\
\langle U\rangle^{\om^\alpha}\treeleq V & \mbox{ if }\alpha>\beta,\\
U \treeleq \langle V\rangle^{\om^\beta} & \mbox{ if }\alpha<\beta.
\end{cases}
\end{align*}
and if $S$ and $T$ are of the form $\langle U\rangle^{\om^\alpha}\mindchange\sqcup_iS_i$ and $\langle V\rangle^{\om^\beta}\mindchange\sqcup_jT_j$, respectively, then
\[
S\treeleq T \iff
\begin{cases}
\mbox{either }  \langle U\rangle^{\om^\alpha}\treeleq \langle V\rangle^{\om^\beta}     &\mbox{and } (\forall i)\;S_i\treeleq T,\\
\mbox{or }\hspace{7mm}      \langle U\rangle^{\om^\alpha}\not\treeleq \langle V\rangle^{\om^\beta}    &\mbox{and }     (\exists j)\;S \treeleq T_j.
\end{cases}
\]
\end{definition}

We now assign a class $\Sigma_T$ to each forest $T\in\forest^{\om_1}(\Q)$ as follows:

\begin{definition}
For $\alpha>0$, $\A\in\PSigma_{\langle T\rangle^{\om^\alpha}}$ if and only if there is a $\mathbf{\Sigma}^0_{1+\om^\alpha}$-measurable function $\D\colon \om^\om\to\om^\om$ and a $\PSigma_T$-function $\B\colon \om^\om\to\Q$ such that $\A=\B\circ\D$.
The other cases are as in Definition \ref{def: Sigma T}.
\end{definition}

We check measurability of $\Sigma_T$-functions.

\begin{lemma}\label{lem:Borel-rank-calc2}
Let $T$ be an $\mathcal{L}(\Q)$-term and $\xi$ be a countable ordinal.
Then, $\PSigma_T\subseteq\mathbf{\Delta}^0_{1+\xi}$ implies $\PSigma_{\langle T\rangle^{\om^\alpha}}\subseteq\mathbf{\Delta}^0_{1+\om^\alpha+\xi}$.
\end{lemma}

\begin{proof}
One can use a similar argument as in Lemma \ref{lem:Borel-rank-calc} (2).
\end{proof}

By combining Lemmas \ref{lem:Borel-rank-calc} and \ref{lem:Borel-rank-calc2}, we obtain the direction from (2) to (1) in Proposition \ref{prop:sublemma1} for infinite Borel rank, that is, that if $T\in \forest^\xi(\Q)$, then $\Sigma_T\subseteq {\bf\Delta}^0_{1+\xi}$.

\subsubsection{$\Sigma_T$-complete functions (infinite Borel rank)}

Now we introduce a $\Sigma_T$-complete function $\Omega_T$ for each forest $T\in\forest^{\om_1}(\Q)$.
To achieve this, we need a universal function at transfinite Borel ranks.
Again, recall that every $\mathbf{\Sigma}^0_{1+\xi}$-measurable function is coded by a real (for instance, we can use Fact \ref{lem:lim-universal}).

\begin{definition}
We say that $\U_\xi\colon \hat{\om}^\om\to\hat{\om}^\om$ is {\em $\mathbf{\Sigma}^0_{\xi}$-universal} if it is $\mathbf{\Sigma}^0_{\xi}$-measurable, and 
for every conciliatory $\mathbf{\Sigma}^0_{\xi}$-measurable function $\Psi\colon \hat{\om}^\om\to\hat{\om}^\om$,
there is a continuous function $\theta$ such that $\Psi=\U_\xi\circ \theta$.
\end{definition}

To show the existence of a $\Sigma_T$-complete function, we need to extend Proposition \ref{prop:universal} to infinite Borel ranks.

\begin{proposition}\label{prop:universal-inf}
For any countable ordinal $\alpha$, there is a $\mathbf{\Sigma}^0_{1+\om^\alpha}$-universal initializable conciliatory function.
\end{proposition}

We prove this proposition in Section \ref{sec:Omega-infinite-Borel-rank}.

We now introduce $\Omega_T$ for each term $T\in\forest^{\om_1}(\Q)$.
It suffices to describe how to define $\Omega_{\langle T\rangle^{\om^\alpha}}$, as the rest is as in Definition \ref{def:universal-function}.

\begin{definition}[Complete Function]
Let $\alpha$ be a countable ordinal, and let $T$ be a tree in $\tree^{\om_1}(\Q)$.
Then we define the conciliatory function $\Omega_{\langle T\rangle^{\om^\alpha}}\colon \om^{\leq\om}\to\mathcal{Q}$ as follows:
\[
\Omega_{\langle T\rangle^{\om^\alpha}}=\Omega_T \circ \U_{\om^\alpha},
\]
where $\U_{\om^\alpha}$ is a fixed $\mathbf{\Sigma}^0_{1+\om^\alpha}$-universal initializable conciliatory function as in Proposition \ref{prop:universal-inf}.
\end{definition}

It is not hard to prove the transfinite versions of \ref{obs:tree-is-conciliatory}, \ref{obs:omega in right pointclass}, \ref{lemma:omega-complete}, and \ref{lem: ordering ->} namely that $\Omega_T$ is $\Sigma_T$-complete, and that if $S\treeleq T$, then $\Omega_S\leq_w\Omega_T$.


\section{The jump operator and its inversion} \label{sec:technical-lemmas}

The goal of this section is to define an inverse of the operation $\B\mapsto \B\circ\U$.
(Recall that $\Omega_{\langle T\rangle}$ was defined as $\Omega_T\circ \U$.)
This operation will be denoted by $\A^{\not\sim}$, and we will prove that $(\A\circ\U)^{\not\sim}\equiv_w \A$ for every $\Q$-valued function $\A$.
Furthermore, we will get that $(A^{\not\sim})\circ \U \equiv_w \A$ if we also assume $\A$ is initializable. 
The key technical notions are the {\em Turing jump operator via true stages} from computability theory and an uniform version of the Friedberg jump inversion theorem.
The use of this jump operator is one of the aspects of our proof that makes it easier than Duparc's work for $\Q=2$.


\subsection{Turing jump operator}\label{sec:Turing-jump}

The Turing jump operator is one of the most basic notions in computability theory.
For our proof, we need a version of this operator with nicer properties than the standard jump operator $X\mapsto X'$.
The jump operator {\em via  true stages} $\J\colon\om^\om\to\om^\om$ introduced by Marcone and Montalb\'an \cite{MMVeblen, Mon14} is exactly what we need.
Marcone and Montalb\'an also defined its approximation on finite strings $J\colon\om^{<\om}\to\om^{<\om}$.
Putting these together, what they defined was a $\Sigma^0_2$ conciliatory function $\J\colon \om^{\leq \om}\to\om^{\leq\om}$.
The properties we need are the following:
\begin{enumerate}
\item ($\Sigma^0_2$-universality from the right.) For every $\Sigma^0_2$ operator $G\colon \om^\om\to \hat{\om}^{\om}$, there is a computable $\theta\colon\J[\om^\om]\to \hat{\om}^{\om}$ such that $G\equiv_{\sf p}\theta\circ \J$ (recall the definition of $\equiv_{\sf p}$ in Section \ref{sec:conciliatory}).
Furthermore, if $G$ is $\Sigma^0_2$ relative to an oracle $C$, then we can still find $\theta$ computable so that,  for every $X\in \om^\om$, $G(X)\mnskip=\theta(\J(C\oplus X))\mnskip$.
\item The image $\om^\om$ under $\J$, namely $\J[\om^\om]$, is a closed subset of $\om^\om$.
Furthermore,  $\J$ is one-to-one, and its inverse $\J^{-1}\colon\J[\om^\om]\to\om^\om$ is continuous. 
\item (Denseness of forcing.) For every string $\gamma\in \om^{<\om}$, there is a string $\si\supseteq \gamma, \si\in \om^{<\om}$ which forces the jump in the following sense: 
We say that $\si\in \om^{<\om}$ {\em forces the jump}, if for every $\tau\supseteq \si$, $\J(\tau)\supseteq \J(\si)$. 
\end{enumerate}

For (1), note that if the range of $G$ is contained in $\om^\om$, one can find $\theta:\J[\om^\om]\to\om^\om$ such that $G=\theta\circ \J$.
The relativized version also holds.
These properties are immediate from the definitions in Marcone and Montalb\'an \cite{MMVeblen, Mon14}.
The last property (3) requires a minute of thought, though it is quite standard.

Notice that for the usual Turing jump operator $X\mapsto X'$, the image is not closed.
Another advantage of $\J$ is that its finite approximation can be easily iterated, allowing us to keep the denseness of forcing when we consider transfinite iterates of the jump.

We use $\J^C$ to denote the operator $X\mapsto \J(C\oplus X)$.
It satisfy the same properties  of $\J$ we mentioned above.
Let $\mathcal{J}^{n,C}$ be the $n$-th iterate of the jump operator relative to $C$, that is, put $\mathcal{J}^{1,C}=\mathcal{J}^C$ and $\mathcal{J}^{n+1,C}=\mathcal{J}^C\circ\mathcal{J}^{n,C}$.
We also use the symbol $\mathcal{J}^{-1,C}$ to denote $(\mathcal{J}^C)^{-1}$.
The $\Si^0_2$ universality of $\J$ can be iterated through the finite levels of the Borel hierarchy.

\begin{fact}\label{fact:limit-lemma}
For every $\mathbf{\Sigma}^0_{n+1}$-measurable function $\A\colon \om^\om\to\hat{\om}^{\om}$, there are $C\in\om^\om$ and a computable $\Phi\colon\om^\om\to\hat{\om}^{\om}$ such that $\A\equiv_{\sf p}\Phi\circ\mathcal{J}^{n,C}$.
\end{fact}

Duparc \cite[Definition 25]{Dup01} introduced an operation $\not\sim$ on subsets of $\om^\om$.
We extend Duparc's operation $\not\sim$ to $\Q$-valued functions, but our definition is quite different from Duparc's, which is rather hard to understand.

\begin{definition}
For any $\A\colon \om^\om\to\mathcal{Q}$ and any oracle $Z\in\om^\om$ we introduce the {\em $Z$-jump inversion} of $\A$, $\A^{\not\sim Z}\colon \mathcal{J}^Z[\om^\om]\to\mathcal{Q}$, as follows:
\[
\A^{\not\sim Z}(Y)=\A(\J^{-1,Z}(Y)),
\]
\end{definition}

Note that the domain of $\A^{\not\sim Z}$ is the range of $\J^Z$, which is closed as mentioned above, and therefore one can think of $\A^{\not\sim Z}$ as a function on $\om^\om$ by Observation \ref{obs:total-wadge}, where we composed $\A^{\not\sim Z}$ with a continuous retraction $\eta\colon\om^\om\to \mathcal{J}^Z[\om^\om]$.

\begin{remark}\label{remark:conciliatory-inversion}
We also apply the operator $\not\sim$ to a conciliatory function $\A:\hat{\om}^\om\to\Q$.
In this case, via a computable homeomorphism $I:\hat{\om}^\om\to\om^\om$, we identify $\A$ with $\A\circ I^{-1}:\om^\om\to\Q$.
Clearly, $\A$ is Wadge equivalent to $\A\circ I^{-1}$.
Then, if the domain of $\A$ is $\hat{\om}^\om$, the actual definition of $\A^{\not\sim Z}$ is $\A\circ I^{-1}\circ\mathcal{J}^{-1,Z}$, and the domain of $\A^{\not\sim Z}$ is $\J^Z[\om^\om]$.
Note that $\A=\A^{\not\sim Z}\circ\mathcal{J}^Z\circ I$.
We should be careful that $\A^{\not\sim}$ is not necessarily conciliatory even if $\A$ is conciliatory.
\end{remark}

\begin{observation}\label{lem:jump-UOP}
 Let $\A$ be any partial $\mathcal{Q}$-valued function, and $Y,Z\in\om^\om$ be oracles.
Then 
\[
Y\leq_TZ
\quad\implies\quad
\A^{\not\sim Y}\geq_w\A^{\not\sim Z}.
\]
\end{observation}

\begin{proof}
The map $X\mapsto\mathcal{J}^Y(X)$ is $\Sigma^0_2$ relative to $Y$, and hence in particular $\Sigma^0_2$ relative to $Z$.
Therefore, there is a computable function $\Phi$ so that $\Phi\circ\mathcal{J}^Z= \mathcal{J}^Y$.
This witnesses that $\A^{\not\sim Z}\leq_w\A^{\not\sim Y}$ as follows: 
For $\J^Z(X)\in \J^Z[\om^\om]$, 
\[
\A^{\not\sim Z}(\J^Z(X))= \A(X) = \A^{\not\sim Y}(\J^Y(X)) = \A^{\not\sim Y}(\Phi(\mathcal{J}^Z(X))).		 \qedhere
\]
\end{proof}

Recall that if $\mathcal{Q}$ is bqo, then so are the Wadge degrees of $\mathcal{Q}$-valued functions (Theorem \ref{thm:Wadge-bqo}).
In particular, there is no infinite decreasing chain of the Wadge degrees.
Thus, Observation  \ref{lem:jump-UOP} implies that, for any sufficiently powerful oracle $Z\in\om^\om$, we have $\A^{\not\sim Z}\leq_w\A^{\not\sim Y}$ for any other oracle $Y\in\om^\om$.

\begin{notation}
If $\A$ is a $\mathcal{Q}$-valued function for a bqo $\mathcal{Q}$, we hereafter use the notation $\A^{\not\sim}$ to denote a representative of the minimum one among Wadge degrees of $\{\A^{\not\sim Z}:Z\in\om^\om\}$, that is,
\begin{align*}
\A^{\not\sim}\equiv_w\A^{\not\sim Z}\mbox{ for some $Z\in\om^\om$}\mbox{, and }\A^{\not\sim}\leq_w\A^{\not\sim Y}\mbox{ for all $Y\in\om^\om$}.
\end{align*}
\end{notation}

%
%

Here are some basic properties of the jump inversion operator.
In particular, (2) of the next lemma shows that $\not\sim$ is well-defined on $\Q$-Wadge degrees.

\begin{lemma}\label{lem:jump-inv}
For any $\A,\B\colon \om^\om\to\mathcal{Q}$, the following holds.
\begin{enumerate}
\item If $\A$ is $\mathbf{\Sigma}^0_{n+1}$-measurable, then $\A^{\not\sim}$ is $\mathbf{\Sigma}^0_n$-measurable.
\item If $\A\leq_w\B$, then $\A^{\not\sim}\leq_w\B^{\not\sim}$.
\item If either $\A^{\not\sim}$ or $\B^{\not\sim}$ is non-self-dual, then $\A^{\not\sim}\leq_w\B^{\not\sim}$ implies $\A\leq_w\B$.   \label{part: sim 1-1}
\end{enumerate}
\end{lemma}

\begin{proof}
(1)
Since $\A$ is $\mathbf{\Sigma}^0_{n+1}$, there is a $\Delta^0_0$ formula $\varphi$ in the language of second-order arithmetic and a $Z\in\om^\om$ such that
\[\A(X)=q\;\iff\;(\exists a_0)(\forall a_1)\dots(\check{\mathsf{Q}}a_{n-1})(\mathsf{Q}a_n)\;\varphi(q,a_0,\dots,a_n,X,Z),\]
where $\mathsf{Q}$ is the existential quantifier if $n$ is even; and $\mathsf{Q}$ is the universal quantifier if $n$ is odd, and $\check{\mathsf{Q}}$ is the other way around.
Clearly, there is a $\Delta^0_0$ formula $\psi$ such that
\[
\A(X)=q\;\iff\;(\exists a_0)(\forall a_1)\dots(\check{\mathsf{Q}}a_{n-1})\;\psi(q,a_0,\dots,a_{n-1},J^Z(X)).
\]
Consequently, $\A^{\not\sim Z}$ is $\mathbf{\Sigma}^0_{n}$.
The same argument works for any $C\geqt C$, and thus  $\A^{\not\sim }$ is $\mathbf{\Sigma}^0_{n}$ too.

(2)
Assume that $\A\leq_w\B$ via a continuous function $\theta$.
Then $\theta$ is $C$-computable for some oracle $C$.
Let $D$ be such that $\A^{\not\sim}\equiv_w\A^{\not\sim D}$ and $\B^{\not\sim}\equiv_w\B^{\not\sim D}$.
Then, by the universality of $\J^D$ from the right, there is a continuous function $\theta'$ such that $\theta'\circ \mathcal{J}^D=\mathcal{J}^D\circ\theta$ since $D\geq_TC$.
This $\theta'$ witnesses that $\A^{\not\sim}\leq_w\B^{\not\sim}$ since for any $Y\in\J^D[\om^\om]$, and $X=\J^{-1,D}(Y)$,
\begin{multline*}
\A^{\not\sim D}(Y)=\A(X)\leq_\mathcal{Q} \\ 
 \B(\theta(X))
	=\B^{\not\sim D}(\mathcal{J}^D(\theta(X)))=\B^{\not\sim D}(\theta'(\mathcal{J}^D(X)))=\B^{\not\sim D}(\theta'(Y)).
\end{multline*}

(3)
If $\A\not\leq_w\B$, then Player I has a winning strategy in the game $G_w(\A,\B)$, that is, there is a $C$-computable Lipschitz function $\theta$ such that $\A(\theta(X))\not\leq_\mathcal{Q}\B(X)$.
By the same argument as above, let $D$ be such that $\A^{\not\sim}\equiv_w\A^{\not\sim D}$ and $\B^{\not\sim}\equiv_w\B^{\not\sim D}$, and then, there is a continuous function $\theta'$ such that $\A^{\not\sim D}(\theta'(Y))\not\leq_\mathcal{Q}\B^{\not\sim D}(Y)$ for all $Y\in\J^D[\om^\om]$.
However, if $\A^{\not\sim D}\leq_w\B^{\not\sim D}$ via a continuous function $\eta$, we would have the following:
\begin{align*}
\A^{\not\sim D}(Z)&\leq_\mathcal{Q}\B^{\not\sim D}(\eta(Z))\not\geq_\mathcal{Q}\A^{\not\sim D}(\theta'\circ\eta(Z)),\\
\B^{\not\sim D}(Z)&\not\geq_\mathcal{Q}\A^{\not\sim D}(\theta'(Z))\leq_\mathcal{Q}\B^{\not\sim D}(\eta\circ\theta'(Z)).
\end{align*}
Thus both $\A^{\not\sim D}$ and $\B^{\not\sim D}$ are self-dual, which contradicts our assumption since $\A^{\not\sim}\equiv_w\A^{\not\sim D}$ and $\B^{\not\sim}\equiv_w\B^{\not\sim D}$.
\end{proof}


\subsection{The operation $\not\sim$ inverts the jump}\label{sec:jump-inversion}

We now prove a key result which is that $\not\sim$ is the inverse of $\B\mapsto \B\circ \U$.
This is somewhat an analogue of \cite[Propositions 29 and 30]{Dup01}, which roughly says that the jump inversion operator $\not\sim$ bridges $\tree^n(\Q)$ and $\tree^{n+1}(\Q)$: namely that $\Omega_{\langle T\rangle}^{\not\sim}\equiv_w\Omega_T$.

Recall from Remark \ref{remark:conciliatory-inversion} that $\Omega_{\langle T\rangle}^{\not\sim}=\Omega_T\circ\U\circ I^{-1}\circ \J^{-1,C}$, where $I:\hat{\om}^\om\to\om^\om$ is a computable homeomorphism.
The following lemma shows how $\U$ and $\J^C$ interact.
It states that $\U\circ I^{-1}\circ \J^{-1,C}$ is Wadge-equivalent to the identity function.

\begin{lemma}\label{lem:UJ-1-continuous}
The map $\U\circ I^{-1}\circ \J^{-1,C}\colon \J^C[\om^\om]\to\hat{\om}^{\om}$ is $\equiv_{\sf p}$-equivalent to a continuous map, and the identity map on $\hat{\om}^\om$ continuously reduces to it, that is, there is a continuous $\theta:\hat{\om}^\om\to\mathcal{J}^C[\om^\om]$ such that $\U\circ I^{-1}\circ \J^{-1,C}\circ \theta\equiv_{\sf p}id$.
\end{lemma}
\begin{proof}
We first see that $\U\circ I^{-1}\circ\J^{-1,C}$ is continuous.
Since $\U\circ I^{-1}:\om^\om\to\hat{\om}^\om$ is ${\Si}^0_2$-measurable, there is a computable $\Phi\colon\om^\om\to\hat{\om}^{\om}$ such that $\U\circ I^{-1}\equiv_{\sf p}\Phi\circ\mathcal{J}^C$.
Then, $\U\circ I^{-1}\circ \J^{-1,C}\equiv_{\sf p}\Phi$ on its domain.

To prove the other direction we need the Friedberg jump inversion theorem \cite{Fr57}.
The standard proof of the Friedberg jump inversion theorem (relative to $C$) gives a $C'$-computable function $\psi$ such that for any $X$,
\[X\oplus C'\equiv_T\psi(X)\oplus C'\equiv_T(\psi(X)\oplus C)'\equiv_T\mathcal{J}^C(\psi(X)).\]
Note that by identifying $\hat{\om}^\om$ and $\om^\om$ via an effective homeomorphism, we can assume that the domain of $\psi$ is $\hat{\om}^\om$ and the range of $\psi$ is included in $\om^\om$ since an effective homeomorphism preserves Turing degrees in a uniform manner.
By carefully checking the standard proof of the Friedberg jump inversion theorem, one can see the following uniform version:
There are $C'$-computable function $\psi\colon \hat{\om}^\om\to\om^\om$ such that, for every $X\in\hat{\om}^{\om}$, $\mathcal{J}^C(\psi(X))$ is uniformly Turing equivalent to $X\oplus C'$.
That is, there is a computable operator $\theta_0\colon\om^\om\to\hat{\om}^\om$ and a $C'$-computable operator $\theta_1\colon\om^\om\to\hat{\om}^\om$ such that, for every $X\in\hat{\om}^{\om}$,
\[
\theta_0\circ\mathcal{J}^C\circ \psi (X) = X
\quad	\mbox{ and }  \quad
\theta_1 (X)=\mathcal{J}^C\circ \psi (X).
\]

Then, on $\hat{\om}^\om$,
\begin{align*}
id	&= 	 \theta_0\circ\mathcal{J}^{C}\circ\psi   \\
			&\equiv_{\sf p}  \U\circ\Phi_2\circ \psi 			\quad\quad &&\mbox{where $\Phi_2$ comes from the universality of $\U$} \\
			&= \U\circ I^{-1}\circ\J^{-1,C}\circ\mathcal{J}^C\circ I\circ\Phi_2\circ \psi \\
			&= \U\circ I^{-1}\circ\J^{-1,C}\circ \theta_2\circ\mathcal{J}^C\circ\psi  \quad\quad &&\mbox{where $\theta_2$ comes from the universality of $\J^C$}\\
			&= \U\circ I^{-1}\circ\J^{-1,C}\circ \theta_2\circ\theta_1  \quad\quad &&\mbox{using that $\mathcal{J}^{C}\circ\psi=\theta_1$}.
\end{align*}
Note that the range of $\Phi_e$ is included in $\hat{\om}^\om$, and therefore, that of $\mathcal{J}^C\circ I\circ\Phi_2$ is in $\mathcal{J}^C[\om^\om]$.
Hence, the range of $\theta_2\circ\theta_1$ is included in $\mathcal{J}^C[\om^\om]$.
Since $\theta_2$ and $\theta_1$ are continuous, $\theta=\theta_2\circ\theta_1$ is the desired continuous function.
\end{proof}

\begin{corollary}\label{cor:jump-cancel} \label{cor:jump-cancel}
Let $\A\colon\hat{\om}^\om\to \Q$ be conciliatory.
Then, 
\[
(\A\circ \U)^{\not\sim}\equiv_w\A.
\]

In particular, if $T\in\tree^n(\mathcal{Q})$, then $\Omega_{\langle T\rangle}^{\not\sim}\equiv_w\Omega_T$.
\end{corollary}
\begin{proof}
All one needs to observe is that $(\A\circ \U)^{\not\sim}= \A\circ\U\circ I^{-1}\circ\J^{-1,C}$.
Then $(\A\circ \U)^{\not\sim}\leq_w\A$ because $\U\circ \J^{-1,C}$ is continuous by Lemma \ref{lem:UJ-1-continuous} and $\A$ is conciliatory.
Moreover, we have $\A\leq_w (\A\circ \U)^{\not\sim}$ via $\theta$ of the previous lemma.

For the second part just notice that $\Omega_{\langle T\rangle}^{\not\sim}=(\Omega_T\circ\U)^{\not\sim}\equiv_w\Omega_T$.
\end{proof}

The operation $\not\sim$ is not one-to-one on $\Q$-Wadge degrees.
But it is if we restrict it to initializable degrees.
We first need to prove the following lemma, which is where the denseness of forcing of $\J$ is needed.

\begin{definition}
For a  function $\A\colon \om^\om\to\Q$, we say that $\A$ is {\em initializable} if for every $\si\in {\om}^{<\om}$, $\A\leq_w \A\upto[\sigma]$.
\end{definition}

That is, $\A$ is initializable if and only if $\F(\A)=\om^\om$.
Recall the definition of $\F$ from Section \ref{sec:Wadge-key-facts}.
Also recall from Proposition \ref{prop:sji-deg-pres} that $\A$ is non-self-dual if and only if $\F(\A)$ is non-empty.

Notice that this definition matches Definition \ref{def: initiazable} for the case $\Q=\om^\om$ were all elements are $\leq_Q$-incomparable.

\begin{lemma}\label{lem:init-join}
If $\A$ is  initializable, then $\A^{\not\sim}$ is non-self-dual.
\end{lemma}

\begin{proof}
Let $C$ be an oracle which computes  Wadge reductions $\A\leq_w\A\upto[\tau]$ for all $\tau$.
We say that {\em $\sigma$ forces its jump relative to $C$} if for any $\tau\supseteq\sigma$, $J^C(\tau)\supseteq J^C(\sigma)$ holds.
We claim that for such $\si$, we have  
\[
\A^{\not\sim C}\leq_w \A^{\not\sim C}\upto [J^C(\sigma)].
\]
To prove this, let $\theta$ be $C$-computable such that $X\mapsto \si\fr \theta(X)$ is a Wadge reduction $\A\leq_w\A\upto[\si]$.
Then, using the universality of $\J^C$ from the right, we have a computable function $\theta'$ such that $\theta'\circ\mathcal{J}^C(X)=\mathcal{J}^C(\sigma\fr\theta(X))$ for all $X\in\om^\om$.
Since $\sigma$ forces its jump relative to $C$, we have 
\[
\J^C(\sigma)\segment\mathcal{J}^C(\sigma\fr\theta(X))=\theta'(\mathcal{J}^C(X)).
\]
Thus $\theta'$ gives a Wadge reduction from $\A^{\not\sim C}$ to $\A^{\not\sim C}\upto [J^C(\sigma)]$ as follows:
Given $Y\in \J^C[\om^\om]$ and $X=\J^{-1,C}(Y)$,
\begin{multline*}
\A^{\not\sim C}(Y)=\A(X)\leq_\mathcal{Q} \\
	 \A(\sigma\fr\theta(X))=\A^{\not\sim C}(\mathcal{J}^C(\sigma\fr\theta(X)))=\A(\theta'(\mathcal{J}^C(X))) =\A(\theta'(Y)).
\end{multline*}

Now, we say that {\em $X\in\om^\om$ forces its jump relative to $C$} if there are infinitely many $n\in\om$ such that $X\upto n$ forces its jump relative to $C$.
Using the density of forcing, one can easily construct such an $X$.
Let $t_0<t_1<t_2<\cdots$ be such that $X\upto t_i$ forces its jump relative to $C$.
We then get that $\J^C(X\upto t_i)\subseteq \J^C(X)$ for all $i$, and furthermore, every initial segment of $\J^C(X)$ is an initial segment of some $\J^C(X\upto t_i)$.
It follows that for every initial segment $\tau\subseteq \J^C(X)$, $\A^{\not\sim C}\leq_w\A^{\not\sim C}\upto [\tau]$.
Therefore, $\mathcal{J}^C(X)\in\F(\A^{\not\sim C})$ as desired.
\end{proof}

\begin{corollary}\label{cor: U and sim inverses}
If $\A$ is conciliatory and initializable, then $\A^{\not\sim}\circ \U\equiv_w \A$.
\end{corollary}
\begin{proof}
By Corollary \ref{cor:jump-cancel} applied to $\A^{\not\sim}$, 
\[
(\A^{\not\sim}\circ \U)^{\not\sim}\equiv_w \A^{\not\sim}.
\] 
By the previous lemma, $\A^{\not\sim}$ is non-self-dual.
But then, by Lemma \ref{lem:jump-inv}(\ref{part: sim 1-1}) applied to the equivalence above, $\A^{\not\sim}\circ \U\equiv_w \A$.
\end{proof}


\subsection{Preservation of ordering}    \label{sec: proof or ordering embedding}
\label{sec:qo-nested-trees}

By a quite straightforward inductive proof, we are now ready to show Proposition \ref{prop:order-on-trees} for finite Borel rank.
Here, by Lemma \ref{lemma:omega-complete}, it suffices to show that $\Omega_S\leq_w\Omega_T$ if and only if $S\treeleq T$ for any $S,T\in\forest^n(\Q)$.
Recall that we have already proved the right-to-left direction in Lemma \ref{lem: ordering ->}.

\begin{proof}[Proof of Proposition \ref{prop:order-on-trees}]
We show the assertion by induction on the terms $S$ and $T$.
Assume $\Omega_S\leq_w\Omega_T$.

First, it is clear that $\Omega_p\leq_w\Omega_q$ if and only if $p\treeleq q$.
Suppose now  that $\Omega_{\langle U\rangle}\leq_w\Omega_{\langle V\rangle}$.
By Corollary \ref{cor:jump-cancel}, $\Omega^{\not\sim}_{\langle U\rangle}\equiv_w\Omega_U$.
We thus get $\Omega_U\leq_w\Omega_V$.
Hence, by the inductive hypothesis,  $U\treeleq V$ and ${\langle U\rangle}\treeleq {\langle V\rangle}$.

Now, consider $S=\langle U\rangle\mindchange\bigsqcup_iS_i$ and $T=\langle V\rangle\mindchange\bigsqcup_jT_j$.
Let us first consider the case that $\langle U\rangle\treeleq\langle V\rangle$.
In this case, by the definition of $\treeleq$, under the assumption that $\langle U\rangle\treeleq\langle V\rangle$, $S\treeleq T$ if and only if $S_i\treeleq T$ for any $i\in\om$.
We get this from the induction hypothesis, as clearly $\Omega_S\leq_w\Omega_T$ implies $\Omega_{S_i}\leq_w\Omega_T$ for any $i$.

Let us now assume that $\langle U\rangle\not\treeleq\langle V\rangle$.
By induction hypothesis, we have $\Omega_{\langle U\rangle}\not\leq_w\Omega_{\langle V\rangle}$.
In this case, by definition, $S\treeleq T$ if and only if $S\treeleq T_j$ for some $j\in\om$.
We thus need to show that $\Omega_S\leq_w\Omega_{T_j}$ for some $j$.

Assume that $\Omega_S\leq_w\Omega_T$ via a continuous function $\theta$.
There must be a sequence $X$ such that $\pi_1(X)$ is empty, but $\pi_1\circ\theta(X)$ is nonempty (that is, $X$ never changes his mind, whereas $\theta(X)$ changes her mind at some point):
This is because if not, $\pi_0\circ\theta$ would give a reduction from $\Omega_{\langle U\rangle}$ to $\Omega_{\langle V\rangle}$, which contradicts our assumption.
Now, let $n$ be the point where $\theta(X\upto n)$ changes her mind, and let $k$ be the first entry of $\pi_1\circ\theta(X\upto n)$.
Then $\pi_1\circ\theta\upto[X\upto n]$ gives a reduction from $\Omega_S\upto [X\upto n]$ to $T_k$.
We now need a reduction from $\Omega_S$ to $\Omega_S\upto [X\upto n]$.

Since $\Omega_{\langle U\rangle}$ is initializable, for any $\sigma$, there is a continuous function $\eta_\sigma$ witnessing $\Omega_{\langle U\rangle}\leq_w\Omega_{\langle U\rangle}\upto[\sigma]$.
It is not hard then, to use $\eta_\sigma$ to build a reduction from $\Omega_S$ to $\Omega_S\upto[\sigma\mindchange\emptyset]$.
This concludes the proof since $\pi_1(X)$ is empty, and thus $X\upto n$ is of the form $\sigma\mindchange\emptyset$.
\end{proof}

We have just proved that the map $T\mapsto \Omega_T$ is an order-preserving embedding of $\forest^n(\Q)$ into the $\Q$-Wadge degrees of $\Delta_{n+1}$-measurable functions.
What is left to do is to show this embedding is onto.

\subsection{Another construction of a $\mathbf{\Sigma}^0_2$-universal function}\label{sec:proof-universal}

We end this section by giving a second construction of a conciliatory $\mathbf{\Sigma}^0_2$-universal function.
The reason we prove this again is that the following proof can be easily extended to through the transfinite, once we define the transfinite jump operation in Section \ref{ss: transfinite jump}.
Recall that our first construction of a conciliatory $\mathbf{\Sigma}^0_2$-universal function in Section \ref{sec:conciliatory} was direct and did not use $\J$.

\begin{proof}[Second proof of Proposition \ref{prop:universal}]
Let $\{\Phi_e:e\in\om\}$ be a computable enumeration of all computable operators $\colon\om^\om\to \om^{\leq \om}$.
By the universality of $\J$ from the right, we know that for every $\bf\Sigma^0_2$ operator $G\colon\om^\om\to \om^{\leq \om}$, there is an $e\in\om$ and a $C\in\om^\om$ such that $G=\Phi_e\circ \J^C$, where $\Phi_e$ is the $e$-th partial computable function.
It is not hard to see that the map
\[
e\fr C\oplus X \mapsto \Phi_e\circ \J^C(X)
\]
is $\bf\Sigma^0_2$ universal (from the left).
However, we need $\U$ to also be initializable, so we need to tweak this definition a bit. 

For $Z\in\hat{\om}^\om$, let $Z^{+1}(n)=Z(n)+1$ if $Z(n)\not={\sf pass}$, and $Z^{+1}({\sf pass})={\sf pass}$.
We define $\U$ as follows
\[
\U(Y)= \begin{cases}  \Phi_e\circ \J^C(X) 	& \mbox{if } Y=\sigma\fr 0\fr Z^{+1} \mbox{ and } Z=e\fr C\oplus X  \\
				\emptyset			& \mbox{if $Y$ has infinitely many $0$'s and is not of the form $\sigma\fr 0\fr Z^{+1}$} \end{cases}
\]

It is clear that $\U$ is still $\bf\Sigma^0_2$ universal.
It is also easily seen to be $\Sigma^0_2$ itself, as deciding in which case we are and where to split $\si$ and $Z$ is $\Sigma^0_2$, but then recovering $e$, $C$ and $X$ is computable. 
It is clearly initializable as $\U(X)=\U(\si\fr 0\concat X)$ for every $\si$ and $X$.
\end{proof}


\section{Proof of ontoness (finite Borel rank)}\label{sec:proof-1-finite}

Recall that we divided Theorem \ref{thm:main-finite-Borel-rank} into Propositions \ref{prop: measurability of Sigma T}, \ref{prop:order-on-trees}, \ref{prop:sublemma1} and \ref{prop:sublemma1(2)->(1)}, and the only one that is left to prove is the latter one.
This whole section is dedicated to proving Proposition \ref{prop:sublemma1(2)->(1)} for finite rank, that is, that given a $\mathbf{\Delta}^0_{1+n}$-measurable function $\A\colon \om^\om\to\Q$, we need to show that there is $T\in\forest^n(\Q)$ such that $\A\equiv_w\Omega_T$.
Furthermore, we will show that if $\A$ is non-self-dual, then $T$ will be a tree, while if $T$ is self dual, $T$ will be a $\sqcup$-type term.

The proof is by induction on $\leq_w$, which we know is well-founded (even bqo). 
We divide the proof in various cases depending on the properties of $\A$.

{\bf Case 1,} {\em Constant}. 
If $\A$ is Wadge equivalent to a constant function, then it is clearly equivalent to $\Omega_{\langle q\rangle}$ for some $q\in\Q$.

{\bf Case 2,} {\em Self Dual}. 
If $\A$ is self dual, we proved in Proposition \ref{prop:sji-deg-pres} that $\A \equiv_w\bigoplus_{i\in\om}\A_i$ where each $\A_i$ is non-self-dual and $\A_i<_w\A$.
By the induction hypothesis, there exists trees $T_i\in\tree^n(\Q)$ such that $\A_i\equiv_w\Omega_{T_i}$.
It then follows that $\A\equiv_w\Omega_{\sqcup_iT_i}$.

These covers the continuous case, as any continuous function a clopen sum of constant functions.

{\bf Case 3,} {\em Initializable}. 
Suppose now that $\A$ is initializable.
We first claim that $\A^{\not\sim}<_w \A$.
That  $\A^{\not\sim}\leq_w \A$ follows from the fact that $\J^{-1,C}$ is continuous. 
Suppose $n$ is the least such that $\A$ is $\Q$-Wadge equivalent to a $\Delta^0_{n+1}$ function.
Since $\A^{\not\sim}$ is $\Delta^0_n$ by Lemma \ref{lem:jump-inv}, it is not $\Q$-Wadge equivalent to $\A$.

Also recall from Lemma \ref{lem:init-join} that $\A^{\not\sim}$ is non-self-dual.
By the induction hypothesis, we then have that there is $T\in \tree^{n-1}(\Q)$ such that $\A^{\not\sim}\equiv_w \Omega_T$.
Moreover, by Corollary \ref{cor: U and sim inverses}, $\Omega_T\equiv_w\Omega_{\langle T\rangle}^{\not\sim}$, and thus $\A^{\not\sim}\equiv_w\Omega_{\langle T\rangle}^{\not\sim}$.
Therefore, by Lemma \ref{lem:jump-inv} (3), we obtain $\A\equiv_w\Omega_{\langle T\rangle}$.

\begin{obs}
Let us comment on the case when the domain of $\A$ is closed subset $\F\subseteq \om^\om$.
In this case we say that $\A$ is initializable if for every $\si\in \om^\om$ extendible in $\F$, $\A\leq_w \A\upto [\si]$.
In Observation \ref{obs:total-wadge}, we notice we could view such a map as a map defined on $\om^\om$ by composing with a retraction $\rho_\F$.
One can show that for the retraction defined there, the map we get is also initializable.
\end{obs}

{\bf Case 4,} {\em Non-self-dual and not initializable}.
Suppose now that $\A$ is non-self-dual and not Wadge equivalent to any initializable function.

\begin{lemma}\label{lem:strategy-remains}
If $\F(\A)\neq \emptyset$, then $\A\upto\F(\A)$ is initializable.

Furthermore, $\A$ is Wadge equivalent to a initializable function, then $\A\equiv_w A\upto \F(A)$.
\end{lemma}
\begin{proof}
Let 
\[
F=\{\si\in\om^{<\om}: \A\leq_w \A\upto[\si]\}.
\]
Then $F$ is a tree without dead ends whose paths are exactly $\F(\A)$.
(We will sometimes write $\si\in \F(\A)$ to mean $\si\in F$.)
We need to show that for each $\si\in F$, there is a continuous reduction $\A\upto \F(\A)\leq_w \A\upto\F(\A)\cap[\si]$.
Since $\si\in F$, there is a continuous reduction $\theta\colon \A\leq_w \A\upto [\si]$.
Think of $\theta$ as a function $\colon\om^{<\om}\to\om^{<\om}$.
We claim that $\theta[F]\subseteq F$, hence obtaining a reduction $\A\upto \F(\A)\leq_w \A\upto\F(\A)\cap[\si]$ as wanted.
Suppose not, and that for some $\tau\in F$, $\theta(\tau)\not\in F$.
We then get a continuous reduction of $\A\upto [\tau]$ to $\A\upto[\theta(\tau)]$.
But $\A\upto [\tau]\equiv_w \A$ while $\A\not\leq_w \A\upto[\theta(\tau)]$, getting the desired contradiction. 

For the second part of the lemma, one needs to observe that if $\A\equiv_w\B$, then $\A\upto \F(\A)\equiv_w\B\upto \F(B)$.
The reason is that if $\theta$ is a reduction $\A\leq_w\B$, then for every $\si\in \F(\A)$, $\theta(\si)$ must be in $\F(\B)$ as $\B\leq_w\A\leq_w\A\upto[\sigma]\leq_w \B\upto [\theta(\sigma)]$.
Now, if $\B$ is initializable, $\F(\B)=\om^\om$, and hence $\A\upto \F(\A)\equiv_w\B\upto \F(B)=\B\equiv \A$.
\end{proof}

Let $V$ be the set of all minimal strings leaving from $\F(\A)$, and then let $\{\sigma_n\}_{n\in\om}$ be an enumeration of $V$.
Then consider $\B=\A\upto\F(\A)$ and $\C=\bigoplus_n(\A\upto[\sigma_n])$.
One can easily check that $\B$ and $\C$ are $\mathbf{\Delta}^0_{1+n}$-measurable whenever $\A$ is.

We first see that $\B,\C<_w\A$.
It is easy to check that $\B,\C\leq_w\A$.
By Lemma \ref{lem:strategy-remains}, $\B$ is initializable, and therefore not Wadge equivalent to $\A$.
To see $\C<_w\A$, we note that $\A\upto[\sigma]<_w\A$ for all $\sigma\in V$.
Therefore, if $\C\equiv_w\A$ then it would imply that $\A$ is $\sigma$-join-reducible, which contradicts our assumption that $\A$ is non-self-dual by Proposition \ref{prop:sji-deg-pres}.
Thus, we get that $\B,\C<_w\A$.

By the induction hypothesis, we get $U\in \tree^{n-1}(\Q)$ and $S\in \forest^n(\Q)$ such that $\B\equiv_w\Omega_{\langle U\rangle}$ and $\C\equiv_w\Omega_S$.

Finally, we claim that $\A\equiv_w \B^\rightarrow\C$.
This would give us that $\A\equiv_w\Omega_{\langle U\rangle\mindchange S}$.
It is straightforward to see that $\A\leq_w\B^\rightarrow\C$.
To see $\B^\rightarrow\C\leq_w\A$, construct a continuous reduction as follows:
Take $X\in\om^\om$, which we write as $\pi_0(X)\mindchange \pi_1(X)$.
While we are reading $\pi_0(X)$ we stay inside $\F(\A)$ using the continuous embedding $\rho_B\colon \B\to\A\upto \F(\A)$.
If we ever change our mind, $\B\mindchange\C(X)=\C(\pi_1(X))$.
We still have that $\rho_B(\pi_0(X))$ is a finite string in $\F(\A)$, and we so can use the reductions $\C\leq_w\A\leq_w \A\upto[\rho_B(\pi_0(X))]$.

More formally: For each $\tau\in\F(\A)$, let $\eta_\tau$ be a Wadge reduction $\A\leq_w \A\upto \tau$.
Let  $\theta_C\colon\C\to\A$.
Then
\[
X\mapsto \left(  \rho_B(\pi_0(X)) \fr \eta_{\rho_B(\pi_0(X))}(\theta_C(\pi_1(X)))\right)
\]
is a reduction $\B^\rightarrow\C\leq_w\A$.


\section{Infinite Borel rank}\label{sec:Omega-infinite-Borel-rank}

We now start to deal with functions of infinite Borel rank.

\subsection{Transfinite jump operator}   \label{ss: transfinite jump}

We now need to consider $\alpha$-th Turing jump for transfinite $\alpha$.
We again use the machinery developed in \cite{MMVeblen,Mon14}.
There, the $\om$-th Turing jump is defined by taking the first bit of each of the finite jumps:
\[
\mathcal{J}^\om(X)=\langle\mathcal{J}(X)(0),\mathcal{J}^2(X)(0),\mathcal{J}^3(X)(0),\dots\rangle.
\]
Because of the way these operators are defined in \cite{MMVeblen,Mon14} taking one bit from each jump is enough to code the whole sequence $\{\J^n(X):n\in\om\}$.
The reason is that $\J^n(X)(0)$ codes a least two bits of $J^{n-1}(X)$, and at least three bits of $J^{n-2}(X)$,..., and at least $n$ bits of $\J(X)$.
See \cite{MMVeblen,Mon14} for more details.

The definition of the $\om^\alpha$-th jump is similar, taking one bit form a sequence of jumps that converges to $\om^\a$.
To make this precise, we need to assign to each countable ordinal $\a$ {\em a fundamental sequence} $(\alpha[n])_{n\in\om}$ which is a non-decreasing sequence with $\a=\sup_n\alpha[n]+1$.
When $\alpha=\beta+1$, we just define $\alpha[n]=\b$, and when $\alpha$ is a limit, $(\alpha[n])_{n\in\om}$ is any increasing sequences towards $\alpha$.
Notice that in either case $\om^\a=\sum_{n}\om^{\alpha[n]}$.
In \cite{MMVeblen,Mon14}, they needed fundamental sequences with particular properties, but we do not need to get into that now.

Something we need in this paper but that was not necessary in \cite{MMVeblen,Mon14} is {\em relative} transfinite Turing jump operators.
If all we needed to do is relativize the whole construction, that would not be any harder.
The problem is that we need to consider transfinite jumps, which are built by iterating jump operators that use different oracles.
For instance, If we deal with the $\om$-th Turing jump, we shall consider a sequence $\C=(C_n)_{n\in\om}$ of oracles to compute the different values $\J^{n,C_n}(X)(0)$.
We call such a sequence an {\em $\om$-oracle}.

In general, we define an $\om^\alpha$-oracle as a sequence $\C=(C_n)_{n\in\om}$ of $\om^{\alpha[n]}$-oracles.
To define the notion of a $\xi$-oracle for ordinals that are not of the form $\om^\a$, note that every ordinal $\xi$ can be written as $\xi=\om^\alpha+\beta$ for some $\beta<\om^{\alpha+1}$ (consider the Cantor normal form).
Then we define a $\xi$-oracle as a pair $\C=(C_0,C_1)$ of an $\om^\alpha$-oracle and a $\beta$-oracle.
Note that for each countable ordinal $\xi$, there is a well-founded tree $\Lambda_\xi\subseteq\om^{<\om}$ such that each $\xi$-oracle $\C$ can be thought of as a $\Lambda_\xi$-indexed collection $(C_\sigma)_{\sigma\in \Lambda_\xi}$.
Given $\xi$-oracles $\C$ and $\D$, we write $\C\leqt\D$ if If $C_\sigma\leq_TD_\sigma$ uniformly holds for any $\sigma\in\Lambda_\xi$.

To introduce the transfinite jump operator, for an $\om^\alpha$-oracle $\mathcal{C}=(C_n)_{n\in\om}$, we use the following abbreviations:
\begin{align*}
\mathcal{J}^{\om^\alpha,\mathcal{C}}_{[0,n)}&=\mathcal{J}^{\om^{\alpha[n-1]},C_{n-1}}\circ\mathcal{J}^{\om^{\alpha[n-2]},C_{n-2}}\circ\dots\circ\mathcal{J}^{\om^{\alpha[0]},C_{0}},\\
J^{\om^\alpha,\mathcal{C}}_{[0,n)}&=J^{\om^{\alpha[n-1]},C_{n-1}}\circ J^{\om^{\alpha[n-2]},C_{n-2}}\circ\dots\circ J^{\om^{\alpha[0]},C_{0}}.
\end{align*}

\begin{definition}[Montalb\'an \cite{Mon14}]
For any countable ordinal $\alpha$ and $\om^\alpha$-oracle $\mathcal{C}$, define the $\om^\alpha$-th jump operation $\mathcal{J}^{\om^\alpha,\mathcal{C}}$ and its approximation $J^{\om^\alpha,\mathcal{C}}$ as follows:
\begin{align*}
\mathcal{J}^{\om^\alpha,\mathcal{C}}(Z)(n)&=\mathcal{J}^{\om^\alpha,\mathcal{C}}_{[0,n+1)}(Z)(0),\\
J^{\om^\alpha,\mathcal{C}}(\sigma)(n)&=J^{\om^\alpha,\mathcal{C}}_{[0,n+1)}(\sigma)(0).
\end{align*}
If a countable ordinal $\xi$ is of the form $\om^\alpha+\beta$ for some $\beta<\om^{\alpha+1}$, for a $\xi$-oracle $\mathcal{C}=({C}_0,{C}_1)$, we define $\mathcal{J}^{\xi,\mathcal{C}}=\mathcal{J}^{\beta,{C}_1}\circ\mathcal{J}^{\om^\alpha,{C}_0}$ and $J^{\xi,\mathcal{C}}=J^{\beta,{C}_1}\circ J^{\om^\alpha,{C}_0}$.
We also use $\J^{-\xi,\C}$ to denote $(\J^{\xi,\C})^{-1}$.
\end{definition}

The property mentioned in the first paragraph in Section \ref{ss: transfinite jump} implies the following well-known fact which connects the transfinite jump operation and the Borel hierarchy.

\begin{fact}[Universality from the right]\label{lem:lim-universal}
If $\A\colon \om^\om\to\hat{\om}^\om$ is $\mathbf{\Sigma}^0_{1+\xi}$-measurable relative to $C$, then there is a computable function $\theta$ such that $\A\equiv_{\sf p}\theta\circ\mathcal{J}^{\xi,C}$.
\end{fact}

(Recall that we write $\A\equiv_{\sf p}\B$ if $\A(X)\mnskip=\B(X)\mnskip$ for all $X\in\mathcal{X}$.)
By using Fact \ref{lem:lim-universal}, it is not hard to construct a $\mathbf{\Sigma}^0_{1+\xi}$-universal initializable conciliatory function (Proposition \ref{prop:universal-inf}) by a similar argument as in Section \ref{sec:proof-universal}.

\begin{proof}[Proof of Proposition \ref{prop:universal-inf}]
By $\Sigma^0_{1+\xi}$-universality of $\mathcal{J}^\xi$ from the right (Fact \ref{lem:lim-universal}), it is obvious that $e\fr C\oplus X\mapsto\Phi_e\circ\mathcal{J}^{\xi,C}(X)$ is $\mathbf{\Sigma}^0_{1+\xi}$-universal from the left.
As in the proof of Proposition \ref{prop:universal} (see Section \ref{sec:proof-universal}), by modifying this function, one can obtain a $\mathbf{\Sigma}^0_{1+\xi}$-universal initializable conciliatory function.
\end{proof}

We now introduce the transfinite version of the jump inversion operator.

\begin{definition}
For any $\A\colon \om^\om\to\mathcal{Q}$ and any oracle $\mathcal{C}$ we introduce the {\em $\om^\alpha$-th jump inversion} of $\A$, $\A^{\not\sim \om^\alpha,\mathcal{C}}\colon \mathcal{J}^{\om^\alpha,\mathcal{C}}[\om^\om]\to\mathcal{Q}$ as follows:
\[\A^{\not\sim\om^\alpha,\mathcal{C}}(X)=\A\circ\mathcal{J}^{-\om^\alpha,\mathcal{C}}(X).\]
\end{definition}

As before, the domain of $\A^{\not\sim \om^\alpha,\mathcal{C}}$ is closed.
If the domain of $\A$ is $\hat{\om}^{\om}$, then $\A^{\not\sim\om^\alpha,\mathcal{C}}$ is defined by $\A\circ I^{-1}\circ\mathcal{J}^{-\om^\alpha,\mathcal{C}}$, where $I:\hat{\om}^\om\to\om^\om$ is a homeomorphism.
The following transfinite version of Observation \ref{lem:jump-UOP} is also straightforward.

\begin{observation}\label{lem:jump-UOP-transfinite}
Let $\A$ be any partial $\mathcal{Q}$-valued function, and $Y$ and $Z$ be $\xi$-oracles.
If $Z\geq_TY$, then $\A^{\not\sim \xi,Z}\leq_w\A^{\not\sim \xi,Y}$.
\end{observation}

Thus, there is $\mathcal{C}$ such that $\A^{\not\sim \om^\alpha,\mathcal{C}}\leq_w\A^{\not\sim \om^\alpha,\mathcal{D}}$ for any $\mathcal{D}$ by well-foundedness of the Wadge degrees of $\Q$-valued functions (Theorem \ref{thm:Wadge-bqo}).
Thus we use $\A^{\not\sim\om^\alpha}$ to denote such $\A^{\not\sim \om^\alpha,\mathcal{C}}$.
By Observation \ref{lem:jump-UOP-transfinite}, $\C$ can be chosen as a constant sequence, that is, $C_\sigma=C$ for any $\sigma\in\Lambda_{\om^\alpha}$.
If $\C$ is a constant sequence consisting of $C\in\om^\om$, we simply write it as $C$ instead of $\C$.
We now see the transfinite version of  Lemma \ref{lem:jump-inv}.

\begin{lemma}\label{lem:basic-transfinite}
For any $\A,\B\colon \om^\om\to\mathcal{Q}$, the following holds.
\begin{enumerate}
\item If $\A$ is $\mathbf{\Sigma}^0_{1+\om^\alpha+\beta}$-measurable, then $\A^{\not\sim\om^\alpha}$ is $\mathbf{\Sigma}^0_{1+\beta}$-measurable.
\item If $\A\leq_w\B$ then $\A^{\not\sim\om^\alpha}\leq_w\B^{\not\sim\om^\alpha}$.
\item If either $\A^{\not\sim\om^\alpha}$ or $\B^{\not\sim\om^\alpha}$ is non-self-dual, then $\A^{\not\sim\om^\alpha}\leq_w\B^{\not\sim\om^\alpha}$ implies $\A\leq_w\B$.
\end{enumerate}
\end{lemma}

\begin{proof}
For (1), let $C$ be such that $\A^{\not\sim\om^\alpha}\equiv_w\A^{\not\sim\om^\alpha,C}$.
If $\A$ is $\mathbf{\Delta}^0_{1+\om^\alpha+\beta}$-measurable, then there is $D\geq_TC$ such that $\A=\Phi_e\circ\mathcal{J}^{\beta,D}\circ\mathcal{J}^{\om^\alpha,D}$ by Fact \ref{lem:lim-universal}.
Then we also have $\A^{\not\sim\om^\alpha}\equiv_w\A^{\not\sim\om^\alpha,D}$, and moreover, $\A^{\not\sim\om^\alpha,D}=\Phi_e\circ\mathcal{J}^{\beta,D}$, which is clearly $\mathbf{\Delta}^0_{1+\beta}$-measurable.
It is straightforward to show (2) and (3) by using the same argument as in the proof of Lemma \ref{lem:jump-inv}.
\end{proof}

To get the transfinite version of Lemma \ref{lem:UJ-1-continuous}, we use the transfinite version of the Friedberg jump inversion theorem \cite{Mac77}:
There exists a $C^{(\om^\alpha)}$-computable function $\psi$ such that
\[X\oplus C^{(\om^\alpha)}\equiv_T\psi(X)\oplus C^{(\om^\alpha)}\equiv_T(\psi(X)\oplus C)^{(\om^\alpha)}\equiv_T\mathcal{J}^{\om^\alpha,C}(\psi(X))\]
holds uniformly in $X$.
Here, as usual, $Z^{(\xi)}$ denotes the $\xi$-th Turing jump of $Z$.
As in the proof of Lemma \ref{lem:UJ-1-continuous}, in particular, there are a computable operator $\theta_0:\om^\om\to\hat{\om}^\om$ and a $C^{(\om^\alpha)}$-computable operator $\theta_1:\om^\om\to\hat{\om}^\om$ such that for every $X\in\hat{\om}^\om$,
\[
\theta_0\circ\mathcal{J}^{\om^\alpha,C}\circ \psi (X) = X
\quad	\mbox{ and }  \quad
\theta_1 (X)=\mathcal{J}^{\om^\alpha,C}\circ \psi (X).
\]

\begin{cor}\label{lem:transfinite-jump-inversion}
$\Omega^{\not\sim{\om^\alpha}}_{\langle T\rangle^{\om^\alpha}}\equiv_w\Omega_T$.
\end{cor}

\begin{proof}
As in the proofs of Lemma \ref{lem:UJ-1-continuous} and Corollary \ref{cor:jump-cancel}, it follows from the above formula.
\end{proof}


\subsection{Generalization of initializability}

\begin{definition}
For a countable ordinal $\alpha$, we say that $\A$ is {\em $\alpha$-stable} if $\A$ is Wadge equivalent to an initializable function, and $\A^{\not\sim\om^\beta}\equiv_w\A$ holds for any $\beta<\alpha$.
\end{definition}

By Lemma \ref{lem:basic-transfinite} (2), if $\B$ is Wadge equivalent to $\A$ and if $\A$ is $\alpha$-stable, then so is $\B$.

\begin{lemma}\label{lem:alpha-stable}
For any $T\in\tree^{\om_1}(\Q)$ and countable ordinal $\alpha$, $\Omega_{\langle T\rangle^{\om^\alpha}}$ is $\alpha$-stable.
\end{lemma}

\begin{proof}
First note that, for any countable ordinal $\alpha$, $\Omega_{\langle T\rangle^{\om^\alpha}}$ is initializable since $\Omega_{\langle T\rangle^{\om^\alpha}}=\Omega_T\circ\U_{\om^\alpha}$ for an initializable function $\U_{\om^\alpha}$.
Recall that $I\colon\hat{\om}^\om\to\om^\om$ is a computable homeomorphism.
Since $I\circ\U_{\om^\alpha}\circ I^{-1}$ is $\mathbf{\Sigma}^0_{1+\om^\alpha}$-measurable, by universality of $\J^{\om^\alpha}$ from the right (Fact \ref{lem:lim-universal}), $\U_{\om^\alpha}=I^{-1}\circ\Phi\circ\mathcal{J}^{\om^\alpha,C}\circ I$ for some $\Phi$ and $C$.
Fix $\beta<\alpha$.
We show that $\Omega_{\langle T\rangle^{\om^\alpha}}$ is Wadge reducible to $\Omega_{\langle T\rangle^{\om^\alpha}}^{\not\sim\om^\beta,C}$.
Let $\psi:\om^\om\to\om^\om$ be the $\om^\beta$-th jump inversion map relative to $C$, that is, there are a computable operator $\theta_0:\om^\om\to\hat{\om}^\om$ and a $C^{(\om^\beta)}$-computable operator $\theta_1:\om^\om\to\hat{\om}^\om$ such that for every $X\in{\om}^\om$,
\[
\theta_0\circ\mathcal{J}^{\om^\beta,C}\circ \psi (X) = X
\quad	\mbox{ and }  \quad
\theta_1 (X)=\mathcal{J}^{\om^\beta,C}\circ \psi (X).
\] 
Then, $\mathcal{J}^{\om^\alpha,C}(Z)=\mathcal{J}^{\om^\alpha,C}\circ\theta_0\circ\mathcal{J}^{\om^\beta,C}\circ \psi (Z)$.
Since $\beta<\alpha$, the map $\mathcal{J}^{\om^\alpha,C}\circ\theta_0\circ\mathcal{J}^{\om^\beta,C}$ is still $\Sigma^0_{1+\om^\alpha}$ relative to $C$.
By universality of $\mathcal{J}^{\om^\alpha}$ from the right (Fact \ref{lem:lim-universal}), there is a computable function $\theta_2$ such that $\theta_2\circ\mathcal{J}^{\om^\alpha}=\mathcal{J}^{\om^\alpha,C}\circ\theta_0\circ\mathcal{J}^{\om^\beta,C}$.
Hence,
\[\mathcal{J}^{\om^\alpha,C}=\theta_2\circ\mathcal{J}^{\om^\alpha,C}\circ\psi.\]

Then,
\[\U_{\om^\alpha}=I^{-1}\circ\Phi\circ\mathcal{J}^{\om^\alpha,C}\circ I=I^{-1}\circ\Phi\circ\theta_2\circ\mathcal{J}^{\om^\alpha,C}\circ\psi\circ I.\]

Clearly, $I^{-1}\circ\Phi\circ\theta_2\circ\mathcal{J}^{\om^\alpha,C}$ (hence, its restriction up to $\om^\om$) is $\Sigma^0_{1+\alpha}$ relative to $C$, and therefore, by universality of $\U_{\om^\alpha}$ from the left, there is a continuous function $\theta_3:\om^\om\to\hat{\om}^\om$ such that
\[\mathcal{U}_{\om^\alpha}=\mathcal{U}_{\om^\alpha}\circ\theta_3\circ\psi\circ I.\]

By our explicit construction of $\mathcal{U}_{\om^\alpha}$ and $C^{(\om^\beta)}$-computability of $\psi$, one can assume that $\theta_3$ is $C^{(\om^{\beta})}$-computable.
By universality of $\mathcal{J}^{\om^{\beta}}$ from the right, there is a computable function $\theta_4$ such that $\mathcal{J}^{\om^{\beta}}\circ I\circ\theta_3=\theta_4\circ\mathcal{J}^{\om^{\beta}}$, and therefore,
\[\mathcal{J}^{\om^{\beta}}\circ I\circ\theta_3\circ\psi\circ I=\theta_4\circ\mathcal{J}^{\om^{\beta}}\circ\psi\circ I=\theta_4\circ\theta_1\circ I.\]

Then, by combining the above formulas, we get
\begin{multline*}
\Omega_{\langle T\rangle^{\om^\alpha}}=\Omega_{T}\circ\mathcal{U}_{\om^\alpha}
=\Omega_{T}\circ\mathcal{U}_{\om^\alpha}\circ\theta_3\circ\psi\circ I
=\Omega_{\langle T\rangle^{\om^\alpha}}\circ\theta_3\circ\psi\circ I\\
=\Omega_{\langle T\rangle^{\om^\alpha}}^{\not\sim\om^\beta,C}\circ\mathcal{J}^{\om^\beta,C}\circ I\circ\theta_3\circ\psi\circ I
=\Omega_{\langle T\rangle^{\om^\alpha}}^{\not\sim\om^\beta,C}\circ\theta_4\circ\theta_1\circ I.
\end{multline*}

Consequently, the continuous function $\theta_4\circ\theta_1\circ I$ gives a Wadge reduction from $\Omega_{\langle T\rangle^{\om^\alpha}}$ to $\Omega_{\langle T\rangle^{\om^\alpha}}^{\not\sim\om^\beta,C}$.
\end{proof}

The purpose of this section is to prove the following transfinite version of Lemma \ref{lem:init-join}.

\begin{lemma}\label{lem:stable-NSD}
If $\A$ is $\alpha$-stable, then $\A^{\not\sim\om^\alpha}$ is $\sigma$-join-irreducible.
\end{lemma}

Before proving Lemma \ref{lem:stable-NSD} we first check that this Lemma immediately implies our main theorems.

\begin{proof}[Proof of Proposition \ref{prop:order-on-trees} from Lemma \ref{lem:stable-NSD}]
We only show that $\Omega_{\langle U\rangle^{\om^\alpha}}\leq_w\Omega_{\langle V\rangle^{\om^\beta}}$ if and only if $\langle U\rangle^{\om^\alpha}\treeleq\langle V\rangle^{\om^\beta}$.
For the other cases, we can use a similar argument as in the proof of Proposition \ref{prop:order-on-trees} for finite Borel rank.
To verify the above equivalence, by Lemma \ref{lem:transfinite-jump-inversion}, we have $\Omega^{\not\sim\om^\alpha}_{\langle U\rangle^{\om^\alpha}}\equiv_w\Omega_U$.
Since $\Omega_{\langle U\rangle^{\om^\alpha}}$ is $\alpha$-stable by Lemma \ref{lem:alpha-stable}, $\Omega^{\not\sim\om^\alpha}_{\langle U\rangle^{\om^\alpha}}$ is $\sigma$-join-irreducible by Lemma \ref{lem:stable-NSD}, and thus non-self-dual by Proposition \ref{prop:sji-deg-pres}.
Thus, if $\alpha=\beta$, by Lemma \ref{lem:basic-transfinite},
\[\Omega_{\langle U\rangle^{\om^\alpha}}\leq_w\Omega_{\langle V\rangle^{\om^\beta}}\iff \Omega_U\equiv_w\Omega^{\not\sim\om^\alpha}_{\langle U\rangle^{\om^\alpha}}\leq_w\Omega^{\not\sim\om^\alpha}_{\langle V\rangle^{\om^\beta}}\equiv_w\Omega_V.\]
This ensures the desired assertion by induction hypothesis.
If $\alpha>\beta$, then, since $\Omega_{\langle U\rangle^{\om^\alpha}}$ is $\alpha$-stable by Lemma \ref{lem:alpha-stable}, we have $\Omega^{\not\sim\om^\beta}_{\langle U\rangle^{\om^\alpha}}\equiv_w\Omega_{\langle U\rangle^{\om^\alpha}}$, and therefore,
\[\Omega_{\langle U\rangle^{\om^\alpha}}\leq_w\Omega_{\langle V\rangle^{\om^\beta}}\iff \Omega_{\langle U\rangle^{\om^\alpha}}\equiv_w\Omega^{\not\sim\om^\beta}_{\langle U\rangle^{\om^\alpha}}\leq_w\Omega^{\not\sim\om^\beta}_{\langle V\rangle^{\om^\beta}}\equiv_w\Omega_V.\]
This ensures the desired assertion by induction hypothesis.
The same argument works in the case $\alpha<\beta$.
\end{proof}

\begin{proof}[Proof of Proposition \ref{prop:sublemma1(2)->(1)} from Lemma \ref{lem:stable-NSD}]
Fix a $\mathbf{\Delta}^0_\xi$-measurable function $\A$.
Let $\delta$ be the smallest ordinal such that $\A$ is not $(\gamma+1)$-stable for some $\gamma<\delta$.
If $\delta=0$, then $\delta$ is not Wadge-initializable, and we can use the same argument as in the proof of Proposition \ref{prop:sublemma1(2)->(1)} for functions of finite Borel rank.

Suppose that $\delta>0$.
Then, note that $\delta$ must be a successor ordinal, say $\delta=\alpha+1$, and thus $\A$ is $\alpha$-stable.
Let $\beta$ be a unique ordinal that $\xi=\om^\alpha+\beta$.
By Lemma \ref{lem:basic-transfinite} (1), $\A^{\not\sim\om^\alpha}$ is $\mathbf{\Delta}^0_\beta$-measurable.
Moreover, by minimality of $\alpha$, we have $\A^{\not\sim\om^\alpha}<_w\A$ and $\A^{\not\sim\om^\alpha}$ is non-self-dual by Lemma \ref{lem:stable-NSD} and Proposition \ref{prop:sji-deg-pres}.
By induction hypothesis, $\A^{\not\sim\om^\alpha}\equiv_w\Omega_T$ for some tree $T\in\tree^\beta(\mathcal{Q})$.
Then we have $\A^{\not\sim\om^\alpha}\equiv_w\Omega_{\langle T\rangle^{\om^\alpha}}^{\not\sim\om^\alpha}$ by Lemma \ref{lem:transfinite-jump-inversion}.
Note that $\Omega_{T}$ is $\sigma$-join-irreducible by Observations \ref{obs:tree-is-conciliatory} and \ref{obs:skip-legal} since $T$ is a tree, and therefore non-self-dual by Proposition \ref{prop:sji-deg-pres}.
Therefore we get $\A\equiv_w\Omega_{\langle T\rangle^{\om^\alpha}}$ by Lemma \ref{lem:basic-transfinite} (3).
Claim that 
\[\langle T\rangle^{\om^\alpha}\in\langle\tree^\beta(\mathcal{Q})\rangle^{\om^\alpha}\subseteq\tree^{\xi}(\mathcal{Q}).\]

It is clear if $\beta<\om^{\alpha+1}$ by definition.
If $\beta\geq\om^{\alpha+1}$, then we must have $\beta=\xi$.
Recall that $\xi$ is of the form $\om^\gamma+\delta$ for some $\gamma<\om_1$ and $\delta<\om^{\gamma+1}$.
We then have $\alpha<\gamma$.
Thus, if $T\in\tree^\xi(\Q)=\tree^\beta(\Q)$, then $\langle T\rangle^{\om^\alpha}\in\tree^\xi(\Q)$.
This concludes the proof.
\end{proof}


\subsection{Proof of Lemma \ref{lem:stable-NSD}}
\label{sec:proof-main-infinite}

It remains to show Lemma \ref{lem:stable-NSD}.
The statement of this lemma resembles Lemma \ref{lem:init-join}, that is, it looks like a transfinite version of Lemma \ref{lem:init-join}.
Nevertheless, our proof requires a very different argument.
The notation for the proof get a bit more complicated than in Lemma \ref{lem:init-join}, as one needs to keep track of $\om$-iterations of the jump.
However, it still is much simpler than Duparc's \cite{Dupxx} proof for $\Q=2$.


\subsubsection{Proof of Lemma \ref{lem:stable-NSD} (for $\alpha\leq 1$)}

Throughout this subsection, for notational simplicity, we always assume that $\B^{\not\sim}\equiv_w\B^{\not\sim\emptyset}$ for any function $\B$.
We will deal with nonempty oracles in the next Section \ref{sec:proof-infinite-general-alpha}.

We first consider the base case $\alpha=0$.
Recall that $\A$ is $0$-stable if and only if it is Wadge equivalent to an initializable function.
Lemma \ref{lem:init-join} states that if $\A$ is $0$-stable, then $\A^{\not\sim}$ is $\sigma$-join-irreducible.
In the proof of Lemma \ref{lem:init-join} we showed that if $\B$ is actually initializable $\si\in \om^{<\om}$ forces its jump, then $\B^{\not\sim}\leq_w \B^{\not\sim}\upto J(\sigma)$, and that if $X\in \om^\om$ forces its jump, then $\J(X)\in \F(\B^{\not\sim})$.
Since every string can be extended to one that forces its jump, the set of $X\in \om^\om$ which force their jump is dense.
We thus get that the set of $X$ such that $\J(X)\in \F(\B^{\not\sim})$ is dense.
In the case when $\A$ is not actually initializable, but Wadge equivalent to an initializable, recall from Lemma \ref{lem:strategy-remains} that $\A\equiv_w\A\upto \F(\A)$.
(Recall Observation \ref{obs:total-wadge} for how to deal with functions whose domain is a closed set as if their domain was all of $\om^\om$.)

Then, the proof of Lemma \ref{lem:init-join} actually implies that
\begin{align}\label{equ:stable-dense0}
\mbox{if $\A$ is $0$-stable, then $\{X: \J(X)\in \F((\A\upto\F(\A))^{\not\sim})$ is dense in $\F(\A)$.}
\end{align}

In the rest of this section, we give a proof of  Lemma \ref{lem:stable-NSD} for $\alpha=1$, that is, if $\A$ is $1$-stable, then $\A^{\not\sim\om}$ is $\sigma$-join-irreducible.
By definition, $\A$ is $1$-stable if and only if $\A$ is Wadge equivalent to an initializable function and $\A^{\not\sim}\equiv_w\A$.
The latter condition is equivalent to that $\A^{\not\sim n}\equiv_w\A$ for any $n\in\om$.
Given a $1$-stable function $\A$, we inductively define a $\mathcal{Q}$-valued function $\A_n$ by 
\[
\A_0=\A
\quad \mbox{ and }\quad
\A_{n+1}=(\A_n\upto\F(\A_n))^{\not\sim}.
\]

\begin{observation}\label{obs:stable-1}
If $\A$ is $1$-stable, then $\A\equiv_w\A_n$ for any $n\in\om$.
\end{observation}

\begin{proof}
Recall that by Lemma \ref{lem:strategy-remains}, a function $\B$ is Wadge equivalent to an initalizable function if and only if $\B\equiv_w\B\upto\F(\B)$.
Fix $n$, and inductively assume that $\A\equiv_w\A_n$.
Therefore, we have $\A\upto\F(\A)\equiv_w\A_n\upto\F(\A_n)$.
Since $\A$ is $1$-stable, in particular, $\A$ is Wadge equivalent to an initializable function.
Then, $\A\equiv_w\A\upto\F(\A)$, and therefore, $\A_n\upto\F(\A_n)$ is $1$-stable.
Thus, $\A_{n+1}\equiv_w\A_n\upto\F(\A_n)$, and therefore $\A_{n+1}\equiv_w\A$.
\end{proof}

In particular, $\A_n$ is initializable for any $n\in\om$.
Thus, the property (\ref{equ:stable-dense0}) implies that
\begin{align}\label{equ:stable-dense1}
\mbox{if $\A$ is $1$-stable, then $\mathcal{J}^{-1}[\F(\A_{n+1})^{\not\sim})]$ is dense in $\F(\A_n)$ for any $n\in\om$.}
\end{align}

For notational convenience, we define 
\[
\F_{n,n}=\F(\A_n)
\quad  \mbox{and}\quad 
\F_{n+1,n}=dom(\A_{n+1}) = \mathcal{J}[\F_{n,n}].
\]
Note that $\F_{n+1,n}$ is the domain of $\A_{n+1,n+1}$, and $\F_{n,n}$ and $\F_{n+1,n}$ are closed sets (since $\mathcal{J}^{-1}$ is continuous).
In the following diagram, the arrow $\hookrightarrow$ indicates the inclusion map.

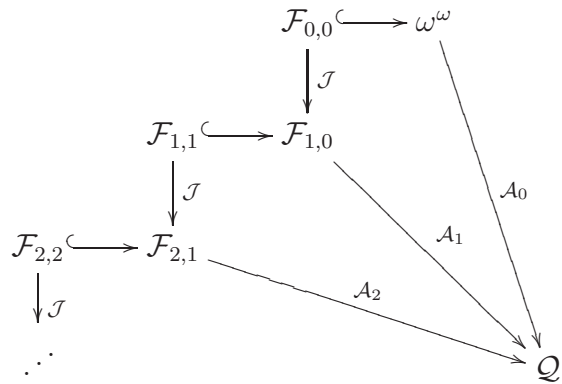
\[
\xymatrix{
	&	& \F_{0,0} \ar@{^{(}->}[r]\ar[d]^{\mathcal{J}}	& \om^\om  \ar[rddd]^{\A_0}		\\
	& \F_{1,1} \ar@{^{(}->}[r]\ar[d]^{\mathcal{J}}	& \F_{1,0}  \ar[rrdd]^{\A_1}		\\
 \F_{2,2} \ar@{^{(}->}[r]\ar[d]^{\mathcal{J}}	& \F_{2,1}  \ar[rrrd]^{\A_2}		\\
\iddots &	&	&		&	\Q
}\]

We also define $\F_{0,1}=\mathcal{J}^{-1}[\F_{1,1}]$ which is included in $\F_{0,0}$, and in general $\F_{n,n+1}=\mathcal{J}^{-1}[\F_{n+1,n+1}]$, which is not necessarily closed.
By using these notations, the property (\ref{equ:stable-dense1}) can be rephrased as:
If $\A$ is $1$-stable, then $\F_{n,n+1}$ is dense in $\F_{n,n}$.
Therefore, we have the following commutative diagram.

\[
\xymatrix{
 & \F_{n,n+1} \;\ar@{^{(}->}[r]^{{\rm dense}} & \F_{n,n} \ar[d]^{\mathcal{J}} \;\ar@{^{(}->}[r] & \F_{n,n-1} \;\ar[rr]^{\A_n} & & \Q\\
\F_{n+1,n+2}  \;\ar@{^{(}->}[r]^{{\rm dense}} & \F_{n+1,n+1} \ar[u]^{\mathcal{J}^{-1}} \;\ar@{^{(}->}[r] & \F_{n+1,n} \;\ar[rrru]_{\A_{n+1}} & & 
}
\]

Generally, for $m<n$, we define $\F_{m,n}=\mathcal{J}^{-1}[\F_{m+1,n}]$.
In particular, $\F_{0,n}=\mathcal{J}^{-n}[\F_{n,n}]$.
This gives a decreasing sequence $(\F_{0,n})_{n\in\om}$.
Define $\F_{0,\om}=\bigcap_{n\in\om}\F_{0,n}$.
In other words,
\[\F_{0,\om}=\{X\in\om^\om:(\forall n\in\om)\;\mathcal{J}^n(X)\in\F_{n,n}\}.\]

Then we define $\F_{\om,\om}=\mathcal{J}^\om[\F_{0,\om}]$.

\[
\xymatrix{
\F_{\om,0} \ar@{^{(}->}[r] \ar[dddd]^{\J^\om}	& \cdots\cdots\cdots\ar@{^{(}->}[r]^{\ \ \ {\rm dense}} 		&\F_{2,0} \ar@{^{(}->}[r]^{{\rm dense}}& \F_{1,0}\ar@{^{(}->}[r]^{{\rm dense}}	& \F_{0,0} \ar@{^{(}->}[r]\ar[d]^{\J}	& \om^\om  \ar[rddd]^{\A_0}		\\
&\cdots\cdots\cdots \ar@{^{(}->}[r]^{\ \ \ {\rm dense}} & \F_{2,1} \ar@{^{(}->}[r]^{{\rm dense}} \ar[u]_{\J^{-1}}	& \F_{1,1} \ar@{^{(}->}[r]\ar[d]^{\J}\ar[u]_{\J^{-1}}	& \F_{1,0}  \ar[rrdd]^{\A_1}		\\
&\cdots\cdots\cdots \ar@{^{(}->}[r]^{\ \ \ {\rm dense}} & \F_{2,2} \ar@{^{(}->}[r]\ar[d]^{\J}	\ar[u]_{\J^{-1}}	& \F_{2,1}  \ar[rrrd]^{\A_2}		\\
&\iddots\iddots&\iddots  \ar@{.>} [rrrr] &	&	&		&	\Q & \\ 
\F_{\om,\om} \ar@{^{(}->}[r]  & \F(\A^{\not\sim\om}) \ar@{.>}^{\A^{\not\sim\om}} [rrrrru] 
}\]

Hereafter, we identify the closed set $\F_{n,n}$ with the pruned tree whose infinite paths are exactly the elements of $\F_{n,n}$.

We devote the rest of this section to prove the following claim:
\begin{align}\label{equ:stable-dense2}
\mbox{If $\A$ is $1$-stable, then $\{X: \mathcal{J}^\om(X)\in  \F(\A^{\not\sim\om})\}$ is dense in $\F(\A)$.}
\end{align}
(Recall that $\F(\A)=\F_{0,0}$.)
Clearly, the claim (\ref{equ:stable-dense2}) entails Lemma \ref{lem:stable-NSD} for $\alpha=1$ as it implies that $\F(\A^{\not\sim\om})\neq\emptyset$.
Our strategy has two steps: 
First to to prove that $\F_{0,\om}$ is dense in $\F(\A)$.
Second to prove that $(\mathcal{J}^\om)^{-1}[\F(\A^{\not\sim\om})]  \supseteq \F_{0,\om} $ by showing that $\F_{\om,\om}\subseteq\F(\A^{\not\sim\om})$ and using that $(\mathcal{J}^\om)^{-1}[\F_{\om,\om}] = \F_{0,\om}$.

\begin{lemma}\label{lem:omega-generic-1}
If $\A$ is $1$-stable, then $\mathcal{F}_{0,\om}$ is dense in $\F_{0,0}$.
\end{lemma}

\begin{proof}
Fix $\sigma\in\F_{0,0}$ and put $\sigma_{0}=\sigma$.
We will construct a sequence $(\sigma_n)_{n\in\om}$ of finite strings such that $\sigma_n\in\F_{n,n}$, and
\[(\mathcal{J}^{n-m})^{-1}(\sigma_n)\segment(\mathcal{J}^{n-m+1})^{-1}(\sigma_{n+1})\in\F_{m,m}\]
for any $m\leq n$.
Then we will define $X:=\bigcup_n(\mathcal{J}^n)^{-1}(\sigma_n)$ and ensure that $X\in\F_{0,\om}$, that is, $\mathcal{J}^n(X)\in\F_{n,n}$.
Given $n$, inductively assume that $\sigma_n\in\F_{n,n}$.
Now, by the property (\ref{equ:stable-dense1}), $\F_{n,n+1}$ is dense in $\F_{n,n}$ for any $n\in\om$.
Since $\F_{n,n+1}=\mathcal{J}^{-1}[\F_{n+1,n+1}]$, there is $Y\in\F_{n+1,n+1}$ such that 
\[\sigma_n\ssegment\mathcal{J}^{-1}(Y)\in\F_{n,n}.\]
Since $\mathcal{J}^{-1}$ is continuous, we can find an initial segment $\sigma_{n+1}\ssegment Y$ such that $\sigma_n\segment\mathcal{J}^{-1}(\sigma_{n+1})$.
Clearly $\sigma_{n+1}\in\F_{n+1,n+1}$.
For every $m\leq n$, by continuity of $(\mathcal{J}^{n-m})^{-1}$, we also have
\[(\mathcal{J}^{n-m})^{-1}(\sigma_n)\segment(\mathcal{J}^{n-m})^{-1}\circ\mathcal{J}^{-1}(\sigma_{n+1})=(\mathcal{J}^{n-m+1})^{-1}(\sigma_{n+1})\]
and $(\mathcal{J}^{n-m})^{-1}(\sigma_n)$ is extendible in $(\mathcal{J}^{n-m})^{-1}[\F_{n,n}]=\F_{m,n}\subseteq\F_{m,m}$, that is, $(\mathcal{J}^{n-m})^{-1}(\sigma_n)\in\F_{m,m}$ as wanted.

For $X=\bigcup_n(\mathcal{J}^n)^{-1}(\sigma_n)$, we claim that $\mathcal{J}^m(X)=Y_m:=\bigcup_{n\geq m}(\mathcal{J}^{n-m})^{-1}(\sigma_n)$.
This is because we have 
\[(\mathcal{J}^m)^{-1}(Y_m)=\bigcup_n(\mathcal{J}^m)^{-1}\circ(\mathcal{J}^{n-m})^{-1}(\sigma_n)=\bigcup_n(\mathcal{J}^{n})^{-1}(\sigma_n)=X.\]
The first equality is due to continuity of $(\mathcal{J}^{m})^{-1}$ and the property that $((\mathcal{J}^{n-m})^{-1}(\sigma_n))_{n\geq m}$ is increasing.
Therefore $\mathcal{J}^m(X)=Y_m$.
Since $(\mathcal{J}^{n-m})^{-1}(\sigma_n)\in\F_{m,m}$, and $\F_{m,m}$ is closed, we have $\mathcal{J}^m(X)\in \F_{m,m}$ for all $m\in\om$, and therefore $\sigma\segment X\in\F_{0,\om}$.
This shows that $\mathcal{F}_{0,\om}$ is dense in $\F_{0,0}$.
\end{proof}

\begin{lemma}\label{lem:omega-initializable-1}
If $\A$ is $1$-stable, then $\F_{\om,\om}\subseteq\F(\A^{\not\sim\om})$.
\end{lemma}

\begin{proof}
Take $Z\in \F_{\om,\om}$; We want to show that $Z\in \F(\A^{\not\sim\om})$.
Given $n\in\om$, we will define a continuous function $\eta:\mathcal{J}^\om[\om^\om]\to\mathcal{J}^\om[\om^\om]\cap[\mathcal{J}^{\om}(X)\upto n]$ witnessing that $\A^{\not\sim\om}\leq_w\A^{\not\sim\om}\upto[Z\upto n]$.

Let $X\in\om^\om=\mathcal{J}^{-\om}(Z)\in \F_{\om,0}$.
Note that $\mathcal{J}^{\om}(X)\upto n=\langle\mathcal{J}(X)\upto 1,\dots,\mathcal{J}^n(X)\upto 1\rangle$.
Recall that $\mathcal{J}^{\om}(X)\in\F_{\om,\om}$ if and only if $\mathcal{J}^n(X)\in\F_{n,n}=\F(\A_n)$ for all $n\in\om$.
Therefore, there is a continuous reduction 
\[
\theta_n\colon \A^{\not\sim n} \leq_w \A_n\upto[\mathcal{J}^n(X)\upto 1])
\] 
since $\A^{\not\sim n}\leq_w \A\leq_w A_n\leq_w A_n\upto [\mathcal{J}^n(X)\upto 1]$,using Observation \ref{obs:stable-1}.
The objective now is to define a continuous function $\eta$ so that it mimics $\theta_n$ when the input is the an $\om$-jump instead of an $n$-jump.
That is, we want $\eta$ so that 
\[
\eta(\mathcal{J}^\om(Y))(k)=\mathcal{J}^k(\mathcal{J}^{-n}\circ\theta_n\circ\mathcal{J}^n(Y))\upto 1.
\]
Since $\mathcal{J}^k(\mathcal{J}^{-n}\circ\theta_n\circ\mathcal{J}^n(Y))$ is clearly uniformly computable from $\mathcal{J}^\om(Y)$, it is not hard to define such function $\eta$. 
In other words, we have that $\eta(\mathcal{J}^\om(Y))=\mathcal{J}^\om(\mathcal{J}^{-n}\circ\theta_n\circ\mathcal{J}^n(Y))$.
Consequently,
\begin{multline*}
\A^{\not\sim\om}(\mathcal{J}^\om(Y))=\A(Y)=\A^{\not\sim n}(\mathcal{J}^n(Y))\leq_\Q\A^{\not\sim n}(\theta_n\circ\mathcal{J}^n(Y))\\
=\A(\mathcal{J}^{-n}\circ\theta_n\circ\mathcal{J}^n(Y))=\A^{\not\sim\om}(\mathcal{J}^\om(\mathcal{J}^{-n}\circ\theta_n\circ\mathcal{J}^n(Y)))=\A^{\not\sim\om}(\eta(\mathcal{J}^\om(Y))).
\end{multline*}
Since $\eta(\mathcal{J}^\om(Y))$ extends $\mathcal{J}^{\om}(X)\upto n$, this witnesses that $\A^{\not\sim\om}\leq_w\A^{\not\sim\om}\upto[\mathcal{J}^{\om}(X)\upto n]$.
\end{proof}

This concludes the proof of Lemma \ref{lem:stable-NSD} for $\alpha=1$.


\subsubsection{Proof of Lemma \ref{lem:stable-NSD} (for general $\alpha$)}\label{sec:proof-infinite-general-alpha}

In this section, we describe the proof of Lemma \ref{lem:stable-NSD} for general $\alpha$, which will be almost no different from the proof for $\alpha=1$.
This section is just for the sake of completeness.
We also explicitly describe how to deal with oracles.

Now, fix a countable ordinal $\alpha$.
By induction, we assume that we have already shown the following claim for any $\beta<\alpha$, if $\A$ is $\beta$-stable, then for any oracle $D$, there is an $\om^\beta$-oracle $C\geq_TD$ such that
\begin{align}\label{equ:stable-dense4}
\mbox{    $\{X: \mathcal{J}^{\om^\beta,C}(X) \in \F((\A\upto\F(\A))^{\not\sim\om^\beta,C})\}$ is dense in $\F(\A)$.}
\end{align}

We now fix an $\alpha$-stable function $\A$.
We will define oracles $(C_n)_{n\in\om}$.
Then, for notational simplicity, we will use the following notations:
\[\J_n=\J^{\om^{\alpha[n],C_n}},\qquad \B^{\not\sim_n}=\B^{\not\sim\om^{\alpha[n],C_n}}.\]

As in the precious section, we inductively define a $\mathcal{Q}$-valued function $\A_n$ and a closed set $\F_{n,n}$ as follows:
\begin{align*}
&\A_0=\A, & &\F_{0,0}=\F(\A).\\
&\A_{n+1}=(\A_n\upto\F_{n,n})^{\not\sim_n},& &\F_{n+1,n+1}=\F(\A_{n+1}).
\end{align*}

To define $\A_{n+1}$, we need to specify oracles $(C_m)_{m\leq n}$.
Before defining these oracles, we introduce several notations.
We define
\begin{align*}
\J_{[m,n)}&=\J_{n-1}\circ \J_{n-2}\circ\dots\circ\J_{m+1}\circ\J_m,\\
\B^{\not\sim}_{[m,n)}&=((\dots(\B^{\not\sim_m})^{\not\sim_{m+1}}\dots)^{\not\sim_{n-2}})^{\not\sim_{n-1}}.
\end{align*}

Note that the sequences defined in the previous section satisfy $\A_n=\A^{\not\sim n}\upto\F_{n,n-1}$.
Now, in our new definition, $\A^{\not\sim n}$ is replaced with $\A^{\not\sim}_{[0,n)}$, that is,
\[\A_n=\A^{\not\sim}_{[0,n)}\upto\F_{n,n-1}\mbox{, where }\F_{n,n-1}=\J[\F_{n-1,n-1}].\]

We now start to define a sequence $(C_n)_{n\in\om}$ of oracles.
Let $C_{-1}$ be an oracle such that $\A^{\not\sim\om^\alpha}\equiv_w\A^{\not\sim\om^\alpha,C_{-1}}$.
Define $\A_0=\A$, and assume that $(C_m)_{m<n}$ are defined, and $\A\equiv_w\A_n$ as in Observation \ref{obs:stable-1}.
In particular, $\A_n$ is $\alpha$-stable.
Then, by induction hypothesis (\ref{equ:stable-dense4}), and initializability of $\A_n$, there is an oracle $C\geq_TC_{n-1}$ such that
\begin{enumerate}
\item[(a)] $(\A_n\upto\F_{n,n})^{\not\sim\om^{\beta[n]},C}\equiv_w(\A_n\upto\F_{n,n})^{\not\sim\om^{\beta[n]}}$.
\item[(b)] $(\mathcal{J}^{\om^{\beta[n]},C})^{-1}[\F((\A_n\upto\F_{n,n})^{\not\sim\om^{\beta[n]},C})]$ is dense in $\F_{n,n}$.
\item[(c)] For any $\sigma\in\F_{n,n}$, there is a $C$-computable Wadge-reduction  $\A^{\not\sim}_{[0,n)} \leq_w \A_n\upto\F_{n,n}\cap[\sigma]$ (recall our proof of Lemma \ref{lem:omega-initializable-1}).
\end{enumerate}

Define $C_n=C$ for such $C$, and then define $\C=(C_n)_{n\in\om}$.
We also define $(\F_{m,n})_{m,n\in\om}$ as in the previous section.
Then, for instance, the above condition (b) can be rephrased as:
$\F_{n,n+1}$ is dense in $\F_{n,n}$.
We then get the following commutative diagram:
\[
{\footnotesize \xymatrix{
\cdots\;\ar@{^{(}->}[r] & \F_{n,n+1} \;\ar@{^{(}->}[r]^{{\rm dense}} & \F_{n,n} \ar[d]^{\mathcal{J}_n} \;\ar@{^{(}->}[r] & \F_{n,n-1} \ar `u_r[rrrr] `_d[rrrr] ^{\A_n} [rrrr] 
\ar[d]^{\mathcal{J}_n} \;\ar@{^{(}->}[r] & \;\cdots \;\;\ar@{^{(}->}[r] & \mathcal{J}_{[0,n)}[\om^\om] \ar[d]^{\mathcal{J}_n}\;\ar[rr]^{\A^{\not\sim}_{[0,n)}} & & \om^\om\\
\cdots\;\ar@{^{(}->}[r]^{\!\!\!\!\!\!\!\!{\rm dense}} & \F_{n+1,n+1} \ar[u]^{\mathcal{J}_n^{-1}} \;\ar@{^{(}->}[r] & \F_{n+1,n} \ar `d^r[rrrrru] `^u[rrrrru] _{\A_{n+1}} [rrrrru]  \;\ar@{^{(}->}[r] & \F_{n+1,n-1} \;\ar@{^{(}->}[r] & \;\cdots \;\;\ar@{^{(}->}[r] & \mathcal{J}_{[0,n+1)}[\om^\om]\;\ar[rru]_{\A^{\not\sim}_{[0,n+1)}} & & 
}}
\]

Define $\F_{0,\om}=\bigcap_{n\in\om}\F_{0,n}$.
In other words,
\[\F_{0,\om}=\{X\in\om^\om:(\forall n\in\om)\;\mathcal{J}^{[0,n)}(X)\in\F_{n,n}\}.\]

Then we define $\F_{\om,\om}=\mathcal{J}^\om[\F_{0,\om}]$.
As in the previous section, we will show the following claim:
\begin{align}\label{equ:stable-dense5}
\mbox{$(\mathcal{J}^{\om^\alpha,\C})^{-1}[\F(\A^{\not\sim\om^\alpha,\C})]$ is dense in $\F(\A)$.}
\end{align}
The claim (\ref{equ:stable-dense5}) entails that $\F(\A^{\not\sim\om^\alpha,\C})$ is nonempty, and therefore $\A^{\not\sim\om^\alpha,\C}$ is $\sigma$-join-irreducible by Proposition \ref{prop:sji-deg-pres}.
Here, since $\C\geq_TC_{-1}$, we have that $\A^{\not\sim\om^\alpha,\C}\equiv_w\A^{\not\sim\om^\alpha}$.
Therefore, the claim (\ref{equ:stable-dense5}) implies that $\A^{\not\sim\om^\alpha}$ is $\sigma$-join-irreducible as desired.
Hence, it suffices to show the claim (\ref{equ:stable-dense5}) to prove Lemma \ref{lem:stable-NSD}.
We will use almost the same strategy as in the previous section.

\begin{lemma}\label{lem:omega-generic-1-general}
$\mathcal{F}_{0,\om}$ is dense in $\F_{0,0}$.
\end{lemma}

\begin{proof}
Fix $\sigma\in\F_{0,0}$ and put $\sigma_{0}=\sigma$.
We will construct a sequence $(\sigma_n)_{n\in\om}$ of finite strings such that $\sigma_n\in\F_{n,n}$, and
\[\mathcal{J}_{[m,n)}^{-1}(\sigma_n)\segment\mathcal{J}_{[m,n+1)}^{-1}(\sigma_{n+1})\in\F_{m,m}\]
for any $m\leq n$.
Then we will define $X:=\bigcup_n\mathcal{J}_{[0,n)}^{-1}(\sigma_n)$ and ensure that $X\in\F_{0,\om}$, that is, $\mathcal{J}_{[0,n)}(X)\in\F_{n,n}$.
Given $n$, inductively assume that $\sigma_n\in\F_{n,n}$.
Now, by the property (b), $\F_{n,n+1}$ is dense in $\F_{n,n}$ for any $n\in\om$.
Since $\F_{n,n+1}=\mathcal{J}_n^{-1}[\F_{n+1,n+1}]$, there is $Y\in\F_{n+1,n+1}$ such that 
\[\sigma_n\ssegment\mathcal{J}_n^{-1}(Y)\in\F_{n,n}.\]
Since $\mathcal{J}_n^{-1}$ is continuous, we can find an initial segment $\sigma_{n+1}\ssegment Y$ such that $\sigma_n\segment\mathcal{J}_n^{-1}(\sigma_{n+1})$.
Clearly $\sigma_{n+1}\in\F_{n+1,n+1}$.
For every $m\leq n$, by continuity of $\mathcal{J}_{[m,n)}^{-1}$, we also have
\[\mathcal{J}_{[m,n)}^{-1}(\sigma_n)\segment\mathcal{J}_{[m,n)}^{-1}\circ\mathcal{J}_n^{-1}(\sigma_{n+1})=\mathcal{J}_{[m,n+1)}^{-1}(\sigma_{n+1})\]
and $\mathcal{J}_{[m,n)}^{-1}(\sigma_n)$ is extendible in $\mathcal{J}_{[m,n)}^{-1}[\F_{n,n}]=\F_{m,n}\subseteq\F_{m,m}$, that is, $\mathcal{J}_{[m,n)}^{-1}(\sigma_n)\in\F_{m,m}$ as wanted.

For $X=\bigcup_n\mathcal{J}_{[0,n)}^{-1}(\sigma_n)$, we claim that $\mathcal{J}_{[0,m)}(X)=Y_m:=\bigcup_{n\geq m}\mathcal{J}_{[m,n)}^{-1}(\sigma_n)$.
This is because we have 
\[\mathcal{J}_{[0,m)}^{-1}(Y_m)=\bigcup_{n\geq m}\mathcal{J}_{[0,m)}^{-1}\circ\mathcal{J}_{[m,n)}^{-1}(\sigma_n)=\bigcup_n\mathcal{J}_{[0,n)}^{-1}(\sigma_n)=X.\]
The first equality is due to continuity of $\mathcal{J}_{[0,m)}^{-1}$ and the property that $(\mathcal{J}_{[m,n)}^{-1}(\sigma_n))_{n\geq m}$ is increasing.
Therefore $\mathcal{J}_{[0,m)}(X)=Y_m$.
Since $\mathcal{J}_{[m,n)}^{-1}(\sigma_n)\in\F_{m,m}$, and $\F_{m,m}$ is closed, we have $\mathcal{J}_{[0,m)}(X)\in \F_{m,m}$ for all $m\in\om$, and therefore $\sigma\ssegment X\in\F_{0,\om}$.
This shows that $\mathcal{F}_{0,\om}$ is dense in $\F_{0,0}$.
\end{proof}

For notational simplicity, we use the following notations:
\[\J_\om=\J^{\om^\alpha,\C},\qquad\A^{\not\sim}_\om=\A^{\not\sim\om^\alpha,\C}.\]

\begin{lemma}\label{lem:omega-initializable-1-general}
If $\A$ is $\alpha$-stable, then $\F_{\om,\om}\subseteq\F(\A^{\not\sim}_{\om})$.
\end{lemma}

\begin{proof}
Fix $X\in\om^\om$ such that $\mathcal{J}_\om(X)\in\F_{\om,\om}$.
Given $n\in\om$, we will define a continuous function $\eta:\mathcal{J}_\om[\om^\om]\to\mathcal{J}_\om[\om^\om]\cap[\mathcal{J}_{\om}(X)\upto n]$ witnessing that $\A^{\not\sim}_\om\leq_w\A^{\not\sim}_\om\upto[\mathcal{J}_{\om}(X)\upto n]$.
Note that $\mathcal{J}_{\om}(X)\upto n=\langle\mathcal{J}_{[0,1)}(X)\upto 1,\dots,\mathcal{J}_{[0,n)}(X)\upto 1\rangle$.
Note also that $\mathcal{J}_{\om}(X)\in\F_{\om,\om}$ if and only if $\mathcal{J}_{[0,n)}(X)\in\F_{n,n}$ for all $n\in\om$.
By the condition (c), Player II has a $C_n$-computable Wadge reduction $\theta_n\colon\A^{\not\sim}_{[0,n)} \leq_w \A_n\upto[\mathcal{J}_{[0,n)}(X)\upto 1]$.
We let $\eta$ by a continuous function such that  for any $k$,
\[
\eta(\mathcal{J}_\om(Y))(k)=\mathcal{J}_{[0,k)}(\mathcal{J}_{[0,n)}^{-1}\circ\theta_n\circ\mathcal{J}_{[0,n)}(Y))\upto 1.
\]
In other words,  $\eta(\mathcal{J}^\om(Y))=\mathcal{J}_\om(\mathcal{J}_{[0,n)}^{-1}\circ\theta_n\circ\mathcal{J}_{[0,n)}(Y))$.
Consequently,
\begin{multline*}
\A^{\not\sim}_\om(\mathcal{J}_\om(Y))=\A(Y)=\A^{\not\sim}_{[0,n)}(\mathcal{J}_{[0,n)}(Y))\leq_\Q\A^{\not\sim}_{[0,n)}(\theta_n\circ\mathcal{J}_{[0,n)}(Y))\\
=\A(\mathcal{J}_{[0,n)}^{-1}\circ\theta_n\circ\mathcal{J}_{[0,n)}(Y))=\A^{\not\sim}_\om(\mathcal{J}_\om(\mathcal{J}_{[0,n)}^{-1}\circ\theta_n\circ\mathcal{J}_{[0,n)}(Y)))=\A^{\not\sim}_\om(\eta(\mathcal{J}_\om(Y))).
\end{multline*}
Since $\eta(\mathcal{J}_\om(Y))$ extends $\mathcal{J}_{\om}(X)\upto n$, this witnesses that $\A^{\not\sim}_\om\leq_w\A^{\not\sim}_\om\upto[\mathcal{J}_{\om}(X)\upto n]$.
\end{proof}

This concludes the proof of Lemma \ref{lem:stable-NSD}.




\bibliography{Wadge_degrees}
\bibliographystyle{alpha}

\end{document}